\documentclass[12pt,a4paper]{amsart}
\usepackage{amsmath,amssymb,amsthm}

\theoremstyle{plain}
\newtheorem{theorem}{Theorem}[section]

\newtheorem{proposition}[theorem]{Proposition}

\newtheorem{definition}[theorem]{Definition}
\newtheorem{example}[theorem]{Example}

\title[Wreath Product Generalizations of  $(S_{2n},H_{n},\xi)$]
{Wreath Product Generalizations of the Triple $(S_{2n},H_{n},\varphi)$ 
and Their Spherical Functions}
\author{Hiroshi Mizukawa}
\date{}
\address{Department of Mathematics, National Defense Academy of Japan,
 Yokosuka 239-8686, Japan.}
\email{mzh@nda.ac.jp}
\begin{document}
\maketitle
\footnote[0]{Classification number :20C05,05E05,05E10,05E35.}
\begin{abstract}
The symmetric group $S_{2n}$ and the hyperoctaheadral group $H_{n}$
is a Gelfand triple for an arbitrary linear representation $\varphi$ of $H_{n}$. 
Their $\varphi$-spherical functions can be caught as 
transition matrix between suitable symmetric functions and 
the power sums.
We generalize this triplet in the term of wreath product. 
It is shown that our triplet are always  a Gelfand
triple.
 Furthermore we study the relation between their spherical functions and multi-partition version of the ring of symmetric 
functions.
\end{abstract}

\section{introduction}
This paper deals with spherical  harmonics on a finite Gelfand triple 
which is a generalization of the triple $(S_{2n},H_{n},\varphi)$,
where $S_{2n}$ is the symmetric group of degree $2n$,
$H_{n}$ is the hyperoctahedral group and $\varphi$ is an arbitrary linear character of $H_{n}$.  
We adopt the terminology "Gelfand triple" by the usage of Bump's book \cite{bump}.
One of main purposes of this paper is to describe an explicit form of the irreducible decomposition of 
their induced representations.  
Furthermore we attempt to express their spherical functions as explicit as possible.
%

%
Roughly speaking,  representation theory of wreath products can be said the
multi-partition version of representation theory of the symmetric groups.
Their irreducible representations and characters can be explicitly computed.
It seems to be no new representational theoretical property of wreath products.  
Nevertheless, from the view point of spherical harmonics, they give us interesting mathematical 
objects like a multivariable version of hypergeometric 
functions as their zonal spherical functions \cite{mz,tanakamiz}.
Furthemore, recent study of finite spherical harmonics related to wreath products succeed in finding new 
orthogonal polynomials and combinatorial or  statistical interpretations of permutation representations \cite{scr,mzonal,scto}.
Therefore it is important to find new finite Gelfand triples and analyze their 
structures as finite homogeneous spaces. 

As mentioned above, the triple $(S_{2n},H_{n},\varphi)$ is a Gelfand triple, i.e. $\varphi\uparrow_{H_{n}}^{S_{2n}}$ 
is a multiplicity free as $S_{2n}$-module.
The characters of Hecke algebra associate with  $(S_{2n},H_{n},\varphi)$ are called $\varphi$-spherical functions. 
It is known that
$\varphi$-spherical functions of this case are deeply related to zonal polynomials
and Schur's $Q$-functions \cite{james,stem}.
To be more explicit, $\varphi$-spherical functions appear in coefficients of  the power sums
in these symmetric functions. 
This fact gives an representation theoretical meaning to these symmetric functions.
Algebraically there exist isomorphisms called characteristic maps between the direct sum of their
Hecke algebras and the ring of symmetric functions.  
In this paper, we consider an generalization of the triplet by changing $S_{2n}$ into 
a wreath product $SG_{2n}=G \wr S_{2n}$. 
We define a subgroup $HG_{n}$ of $SG_{2n}$ which is a natural generalization of $H_{n}$
 (details are in Section 2).
In \cite{mzonal}, the author shows the triplet $(SG_{2n},HG_{n},1)$
is Gelfand triple. 
When the Frobenius-Schur indicator of 
all irreducible characters of $G$ is equal to $0$
or 1,
 their spherical functions  are
obtained as the coefficients of multi-partition version of power sum in product of
Schur functions and zonal polynomials. 
Here we have to explain multi-partition version of symmetric 
functions. 	
For a finite set $A$, we prepare the power sums $p_{r}(a)$ $(a \in A)$ and set
$\Lambda[A]={\mathbb C}\langle p_{r}(a) \mid a \in A, r \geq 1\rangle$.
We call $\Lambda[A]$ the multi-partition version of the ring of symmetric 
functions associated with $A$.
In the book \cite{mac},  the character theory of wreath products is introduced 
in terms of $\Lambda[G^{*}]$, where $G^*$ denotes the set of the irreducible representations 
of a finite group $G$. Recently  Ingram,  Jing and  Stitzinger study this type of symmetric functions  in \cite{xi}. 

One of our main result of this paper is to show that  a triplet $(SG_{2n},HG_{n},\Theta_{\xi,\pi})$ 
is a Gelfand triple for an arbitrary linear character $\Theta_{\xi,\pi}$ of $HG_{n}$.
Here all linear representations of $HG_{n}$ can be indexed by 
 linear representations $\xi$ of $G$ and $\pi$ of $H_n$ (see. Section 4.1). 
The irreducible representations of $SG_{2n}$ are indexed by a $|G^*|$-tuple of partitions
$(\lambda(\chi)\mid \chi \in G^*)$. 
The problem we should consider is:
 which partitions appear as irreducible component of  $\Theta_{\xi,\pi} \uparrow_{HG_{n}}^{SG_{2n}}$?
 To answer this problem, we need to show a twisted version Frobenius-Schur theorem:
 $$\nu_{2}^{\xi}(\chi)=\frac{1}{|G|}\sum_{g \in G}\overline{\xi(g)}\chi(g^2)=-1,0,1\ \ ({\rm Theorem}\ \ref{modFS}).$$
 We see that these three values 1, 0 and -1 determine the type of a partition $\lambda(\chi)$
 in a irreducible component  indexed by $(\lambda(\chi) \mid \chi \in G^*)$ of the induced representation.
After completing this problem, we consider the $\Theta_{\xi,\pi}$-spherical functions  of  $(SG_{2n},HG_{n},\Theta_{\xi,\pi})$.
We are greatly interested in $\Theta_{\xi,\pi}$-spherical
 function because they gives sometimes important examples of orthogonal functions. 
Actually our  $\Theta_{\xi,\pi}$-spherical functions appear as coefficients of multi-power sum of product including Schur functions, Schur's $Q$-functions and Jack polynomials at the parameter $\alpha=2$ and $1/2$.
 
 This paper is organized as follows.
We prepare the almost all notations used in this paper in Section 2.
In section 3, we consider a twisted version  of the Frobenius-Schur formula through
an analysis of $(SG_{2},HG_{1},\Theta_{\xi,\pi})$.
As conclusion of this section, we have a certain relation between
the number of irreducible representations of finite groups and that of conjugacy classes.
Section 4 is devoted to description of the irreducible decomposition of our triplet.
Simultaneously we discuss basis of the Hecke algebra.
 The main tool of this section is the relation obtained in Section 3.
In section 5, we consider  spherical functions. Two special cases  are 
computed explicitly.  
In section 6, we obtain the relation between $\Theta_{\xi,\pi}$-spherical functions
and multi-partition version of symmetric functions.
We construct an graded algebra which are the direct sum of Hecke algebra.
We make an isomorphism $CH_{\pi}$ between the graded Hecke algebra and 
a certain multi-partition version of the symmetric functions. 
We see that  the image of $\Theta_{\xi,\pi}$-spherical functions under $CH_{\pi}$ is
products of some symmetric functions.
In the last section, we apply the spherical harmonics of our triplet to the case of $G={\mathbb Z}/2{\mathbb Z}$ and see a their 
$\Theta_{\xi,\pi}$-spherical functions are essentially same as the characters of symmetric groups.

\section{Preliminaries}
Let $G$ be a finite group.
We denote by
$G_{*}$  (resp. $G^*$) the set of the  conjugacy classes 
(resp. a complete representatives of the irreducible representations) of $G$.
Let ${\mathcal L}(G)$ be the set of all linear characters of $G$.
For  $g \in G$, $C_{g}$ denotes the conjugacy class of $g$. Put $\zeta_{g}=\frac{|G|}{|C_{g}|}$.
For $\gamma \in G^{*}$, we denote by $\chi_{\gamma}$ the character of a representation $\gamma$.
We always identify $\{\chi_{\gamma}| \gamma \in G^{*}\}$ with $G^{*}$.
Fix $\eta \in {\mathcal L}(G)$.
Put $G_{**}=\{C_{g} \cup C_{g^{-1}}| g \in G \}$ and
 $G^{**}_{\eta}=\{\rho | \rho\in G_{*}\}/_{\sim_{\eta}}$, where,  
 $$\rho \sim_{\eta} \sigma \Leftrightarrow \chi_{\rho}={\chi_{\sigma}}  \ {\rm or} \ \chi_{\rho}=\overline{\chi_{\sigma}}\otimes
 \eta \ \ ( \rho,\sigma \in G^*).$$
By fixing a  representatives, we identify $G_{\eta}^{**}$ with a subset of $G^{*}$.
 Put $n_{**}=|G_{**}|$ and $n_{**}^{\eta}=|G^{**}_{\eta}|$. We remark $n^{**}=n_{**}^{1}$,
 where 1 is the trivial representation of $G$. 
If $C_{g}=C_{g^{-1}}$ (resp. $C_{g} \not =C_{g^{-1}}$),
 then we call $C_{g}$ and $R_g=C_g \cup C_{g^{-1}}$  a {\it  real}  (resp. {\it a complex}).
 \begin{example}
 In the case of $G=C_{6}=\langle a\mid a^6=1\rangle$ , we set $\xi_{m}:a \mapsto \exp\frac{2m \pi i}{6}
 \ (0 \leq m \leq 5)$.
 Then  we have
 $G_{**}=\{\{1 \},\{a,a^5 \},\{a^2,a^4 \},\{a^3\}\}$ and $G_{\xi_{1}}^{**}=\left\{\{\xi_0,\xi_{1}\},\{\xi_2,\xi_5\},\{\xi_3,\xi_4\}\right\}$.
 \end{example}
Denote by $\mathbb{C}G$ the group algebra of $G$.
We identify  $f=\sum_{x \in G}f(x)x \in \mathbb{C}G$ with the function $x \mapsto f(x).$ 
Under this identification, the multiplication in $\mathbb{C}G$ is given by the convolution:
$(f_{1}f_{2})(x)=\sum_{y \in G}f_{1}(y^{-1})f_{2}(yx)$.
%
%
Let $H$ be a subgroup of $G$. For $\xi \in {\mathcal L}(H)$,
let  $e_{\xi}=\frac{1}{|H|}\sum_{h \in H}\xi(h^{-1})h$.
A subalgebra
$ {\mathcal H}^{\xi}(G,H)=e_{\xi}{\mathbb C}Ge_{\xi}$ of 
$\mathbb{C}G$ is called a {\it Hecke algebra} of a triplet $(G,H,\xi)$.
The Hecke algebra ${\mathcal H}^{\xi}(G,H)$ can be identified with
$$\{ f:G \rightarrow \mathbb{C} \mid f(gh)=f(hg)=\overline{\xi(h)}f(g)\ \  (\forall g \in G, \forall h \in H)\}.$$
We denote by  $\xi \uparrow_{H}^{G} $
a representation of $G$
 induced from $\xi$.
If  $\xi\uparrow ^{G}_{H}$ is multiplicity-free as $G$-module, then we call 
the triplet $(G,H,\xi)$ a {\it{Gelfand triple}}. 
From Schur's lemma, 
`` $\xi\uparrow^{G}_{H}$ is multiplicity-free"
and
 ``$\mathcal{H}^{\xi}(G,H)$ is commutative"
are equivalent.
%
 In this paper, we only consider representations of finite groups.
Therefore  we  assume that each representation space $V$ of $G$ has $G$-invariant inner product
$\langle ,\rangle_{V}$. 
We assume that $(G,H,\xi)$ is a Gelfand triple. 
Let $\gamma \in G^*$ be an irreducible component of  $\xi\uparrow_{H}^{G}$.
Then there is a unique element of $\gamma$ satisfying
$$h v_0=\xi(h)v_0\ {\rm and}\ \langle v_0,v_0 \rangle_{\gamma} =1.$$
We define a {\it $\xi$-spherical function} by
$$\omega_{\gamma}(g)=\langle v_0,gv_0 \rangle_{W}.$$
 Also {\it $\xi$-spherical function} has the following expression;
 $$\omega_{\gamma}=\frac{|G|}{\dim{\gamma}}e_{\gamma}e_{\xi},$$
 where $e_{\gamma}=\frac{\dim \gamma}{|G|}\sum_{g \in G} \chi_{\gamma}(g^{-1})g$.
 
Let $\lambda$ be a partition and let $P_{n}$ be the set of all partitions of $n$.  
We denote by $\lambda'$ the transpose of $\lambda$.  
For $\lambda$, $m_{i}(\lambda)$ denotes the multiplicity of $i$ in $\lambda$.
Let $\ell(\lambda)$ be the length of $\lambda$.
 If all parts of $\lambda$ are even (resp. odd), then we call it a even (resp. odd) partition.
 We denote by $EP_{n}$ (resp. $OP_{n}$) the set of all even (resp. odd) partitions of $n$.  
Let $SP_{n}$ be the set of all strict 
partitions of $n$. Set $P =\cup_{n \geq 0} P_{n}$, $SP =\cup_{n \geq 0} SP_{n}$,
$OP =\cup_{n \geq 0} OP_{n}$ and 
$EP =\cup_{n \geq 0} EP_{n}$.  We denote by $D(\mu)$ the doubling of $\mu \in SP$ (cf. \cite{stem}).
For example, $D(421)=(5441)$.
Let $\underline{\lambda}=(\lambda(x)|x \in X)$ be a $|X|$-tuple of partitions 
for a set $X$. We define the weight of $\underline{\lambda}$ by $|\underline{\lambda}|=\sum_{x \in X}|\lambda(x)|$.
For two partitions $\lambda$ and $\mu$, we define a partition $\lambda \cup \mu$ by $m_{i}(\lambda \cup \mu)
=m_{i}(\lambda)+m_{i}(\mu)$ for any $i$. 
For $\underline{\lambda}=(\lambda(x)|x \in X)$  and $\underline{\mu}=(\mu(x)|x \in X)$,
set $\underline{\lambda} \cup \underline{\mu}=(\lambda(x)\cup \mu(x)|x \in X)$. 
For $(\rho(x)\mid x \in X)$,
 we define $\hat{\underline{\rho}}$ to be a partition $\bigcup_{x \in X}\rho(x)$.

  The hyperoctahedral group $H_{n}$ is a subgroup of $S_{2n}$ defined by
 $$H_{n}\cong \langle (2i-1,2i), (2j-1,2j+1)(2j,2j+2) \mid 1 \leq i \leq n,1 \leq j \leq n-1\rangle.$$
 $H_{n}$ is the centralizer of an element $(12)(34)\cdots(2n-1,2n) \in S_{2n}$.
 Here we remark  $\langle(2j-1,2j+1)(2j,2j+2) \mid 1 \leq i \leq n,1 \leq j \leq n-1\rangle \cong S_{n}=
 \langle(i,i+1)|1 \leq i \leq n\!-\!1 \rangle.$
 We define an isomorphism $\phi_{n}$ from $S_{n}$ onto $\langle(2j-1,2j+1)(2j,2j+2) | 1 \leq i \leq n,1 \leq j \leq n-1\rangle
 \subset H_{n}$ by
 $$\phi_{n}((i,i+1)) = (2i-1,2i+1)(2i,2i+2).$$
We consider the wreath product  $SG_{n}=G \wr S_{n}=\{(g_{1},g_{2},\cdots,g_{n}:\sigma)\mid g_{i} \in G, \sigma \in S_{n}\}$
of a finite group $G$ with a symmetric group $S_{n}$.
  We define a subgroup $HG_{n}$ of $SG_{2n}$  by
 $$HG_{n}=\{(g_{1},g_{1},g_{2},g_{2},\cdots,g_{n},g_{n}:\sigma)\mid \sigma \in H_{n}\}.$$
Put  $\underline{\lambda}=(\lambda(\chi)\mid \chi \in G^*)$ and $|\underline{\lambda}|=n$.
 We denote by $S(\underline{\lambda})$ the irreducible representation of $SG_{n}$ indexed
 by $\underline{\lambda}$ which is constructed by the following way (cf. 
 \cite{JK}). 
  Let $V_{\chi}$ be an irreducible representation of $G$  affording a character $\chi$ and 
 $S^{\lambda}$ be an irreducible representation of $S_{n}$ indexed by a partition $\lambda$.
We can define an action of $SG_{n}$ on $S^{\lambda}(\chi)=V_{\chi}^{\otimes n} \otimes S^{\lambda}$ by
 $$(g_{1},\cdots,g_{n}:\sigma)v_{1}\otimes \cdots \otimes v_{n} \otimes w=g_{1}v_{\sigma^{-1}(1)}
 \otimes \cdots \otimes g_{n}v_{\sigma^{-1}(n)} \otimes \sigma w,$$
 where $v_{i} \in V_{\chi}$ and $w \in S^{\lambda}$.
 Set $n_{\chi}=|\lambda(\chi)|$ and $\underline{n}=(n_{\chi}\mid \chi\in G^*)$.
 We have $S(\underline{\lambda}) \cong \bigotimes_{\chi \in G^*} S^{\lambda(\chi)}(\chi)
 \uparrow_{SG(\underline{n})}^{SG_{n}}$, where $SG(\underline{n}) =\prod_{\chi \in G^*} SG_{n_{\chi}}$.
 It is a fact that $\{S(\underline{\lambda}) \mid |\underline{\lambda}|=n\}$ gives a complete representatives of 
 the irreducible representations of $SG_{n}$.

 In  \cite{mzonal}, the author shows the following theorem.
 \begin{theorem}\label{mzprev}\cite{mzonal}
\begin{enumerate}
\item
 For any $x \in SG_{2n}$, we have $HG_{n}xHG_{n}=HG_{n}x^{-1}HG_{n}$.
 \item
  The double cosets $HG_{n} \backslash SG_{2n}/HG_{n}$ are indexed by
 $|G_{**}|$-tuple of partitions $\underline{\rho}=(\rho(R)|R \in G_{**})$ such that 
 $\sum_{R \in G_{**}}|\rho_{R}|=n$.
\end{enumerate}
\end{theorem}
We define some notations of special  elements of the symmetric groups and wreath products. 
After the definition, we see examples of these symbols.
\begin{definition}
\begin{enumerate}
\item
For $x_{i} \in S_{n_{i}}$ $(1 \leq i \leq m)$, We define 
a parmutation $[x_{1},x_{2},\cdots ,x_{m}]_{n_{1},\cdots,n_{m}}$ to be a 
natural embedding of $x_{1} \times \cdots \times x_{m}$ in $S_{n_1+\cdots+n_{m}}$.
\item
For a partition $\rho=(\rho_{1},\cdots,\rho_{\ell(\rho)}) \in P_{n}$, we define a parmutation 
$[\rho]$
by $[\rho]=[\rho_{1},\cdots,\rho_{\ell(\rho)}]_{\rho_{1},\cdots,\rho_{\ell(\rho)}}$.
\item
In a similar way of 1,
for $X_{i} \in SG_{n_{i}}$ $(1 \leq i \leq m)$,  we define an element $[X_{1},\cdots,X_{m}]_{n_{1},\cdots,n_{m}}$
to be a natural embedding of $X_{1} \times \cdots \times X_{m} \in SG_{n_{1}}\times \cdots 
\times SG_{n_{m}}$ in $SG_{n_{1}+\cdots +n_{m}}$.
\end{enumerate}
\end{definition}
\begin{example}
\begin{enumerate}
\item
If $x=(12)\in S_{3}$, $y=(13)(24) \in S_{4}$ and $z=(12) \in S_{2}$,  then
$[x,y,z]_{3,4,2}=(12)(46)(57)(89) \in S_{9}$.
\item
For a partition $\rho=(4,2,2)$, we have
$[(4,2,2)]=(1234)(56)(78) \in S_{8}$.
\item
If $X=(g_{1},g_{2},g_{3}:(12))\in SG_{3}$, $Y=(g_{4},g_{5},g_{6},g_{7}:(13)(24))
 \in SG_{4}$ and $Z=(g_{8},g_{9}:(12)) \in SG_{2}$,  then
$[X,Y,Z]_{3,4,2}=(g_{1},\cdots,g_{9}:(12)(46)(57)(89) )\in S_{9}$.
\end{enumerate}
\end{example}
Under these notations,  we can choose the complete representative of $HG_{n} \backslash SG_{2n}/HG_{n}$ 
as follows.
 \begin{theorem}\label{doublecoset}\cite{mzonal}
Set $G_{**}=\{R_{1},R_{2},\cdots,R_{s}\}$ and fix an element $g_{R_{i}}$ of $R_{i}$ $(1 \leq i \leq s)$.
Let  $\rho(R_{i})=(\rho_{1}(R_{i}),\rho_{2}(R_{i}),\cdots)$ be a partition of $n_{i}$. Put
$$x(R_{i})=
\begin{cases}
(\underbrace{1,\cdots 1,g_{R_{i}}}_{2\rho_{1}(R_{i})},
\underbrace{1,\cdots 1,g_{R_{i}}}_{2\rho_{2}(R_{i})},
\cdots,
\underbrace{1,\cdots 1,g_{R_{i}}}_{2\rho_{\ell(\rho(R_{i}))}(R_{i})},
:[2\rho(R_{i})]
), & \rho(R_{i})\not= \emptyset\\
\emptyset,& \rho(R_{i})= \emptyset.
\end{cases}
$$
If $n_{1}+\cdots n_{s}=n$, then each complete representative of $HG_{n} \backslash SG_{2n}/HG_{n}$  corresponding to
$\underline{\rho}=(\rho(R)|R \in G_{**})$
can be chosen by the following form
$$x(\underline{\rho})=[x(R_{1}),\cdots,x(R_{s})]_{2n_{1},\cdots,2n_{s}}.$$
\end{theorem}
If $G$ is the trivial,  then the set $\{[2\rho]\mid \rho \in P_{n}\}$ is a complete representatives 
of $H_{n} \backslash S_{2n}/H_{n}$ (cf. \cite{james}). 
\begin{example}
We consider $G=C_{4}=\{e,a,a^2,a^3\}$ and $n=4$.
Then we have $G_{**}=\{R_{1}=\{0\},R_{2}=\{1,3\},R_{3}=\{2\}\}$.
We consider a representative of  $HG_{4} \backslash SG_{8}/HG_{4}$ corresponding to $(\emptyset,(1),(2,1))$.
We have $x(R_{1})=\emptyset$, $x(R_{2})=(e,a:(12))$ and $x(R_{3})=(e,e,e,a^2,e,a^2:(1234)(56))$ .
Therefore we have
$$[\emptyset,(e,a:(12)),(e,e,e,a^2,e,a^2:(1234)(56))]_{0,2,6}
=(\underbrace{e,a}_{2},\underbrace{e,e,e,a^2,e,a^2}_{4,2}
:(12)(3456)(78)).$$
\end{example}

\section{Twisted Frobenius-Schur's Theorem}
Let $G$ be a finite group. 
The following identity is known as the Frobenius-Schur's theorem: 
$$\frac{1}{|G|}\sum_{g\in G} \chi(g^2)=1,-1\ {\rm or} \ 0.$$
Fix   $\xi \in {\mathcal L}(G)$.
In this section, we consider  the following summation:
$$\nu_{2}^{\xi}(\chi)=\frac{1}{|G|}\sum_{g\in G} \overline{\xi (g)}\chi(g^2).$$

We consider a pair $(SG_{2},HG_{1})$. This pair is a Gelfand pair (cf. \cite{mzonal}). 
Fix a complete representatives   $\{ g_{R}|R \in G_{**}\}$ of $G_{**}$.
A complete representatives of their double coset can be chosen by
$$\{(1,g_{R};1)|R \in G_{**}\},$$
where we remark  $(1,g_{R};1)=(1,g_{R};(12))(1,1;(12))$ (cf. Theorem \ref{doublecoset}).
We remark $HG_{1}\cong G \times S_{2}$, therefore any 
element of ${\mathcal L}(HG_{1})$ can be written as
$$(\xi\otimes \varepsilon)(g,g;\sigma)=\xi(g)\varepsilon(\sigma),$$
where $\varepsilon$ is the trivial or sign representation of  $S_2$.

In unitary representation theory of finite groups,
the following proposition is well known.
\begin{proposition}\label{Kondo}
\begin{enumerate}
\item
Let $R_{1}$ and $R_{2}$ be unitary representations of $G$.
If $R_{1}\sim R_{2}$, then there exists a unitary matrix $S$ such that
$SR_{1}S^{-1}=R_{2}$.
\item
Let  $R$ be a unitary representation of $G$.
A matrix $U$ satisfy 
$$UR(x)U^{-1}=\overline{R}(x),$$
if and only if
$$\begin{cases}
U={}^{t}U & (\text{R\ is\ a\ type I}\ ),\\
U=-{}^{t}U & (\text{R\ is\ a\  type II}\ ).
\end{cases}$$
\end{enumerate}
\end{proposition}

%
\begin{proposition}
Let $\chi$ be an irreducible character of $G$.
Then we have
\begin{align*}
\frac{1}{|G|}\sum_{x \in G} \overline{\xi(x)} \chi (x^2)=
\begin{cases}
\alpha & (\chi = \overline{\chi} \otimes {\xi})\\
0 & (\chi \not= \overline{\chi} \otimes {\xi}),
\end{cases}
\end{align*}
where $|\alpha|=1$.
\end{proposition}
\begin{proof}
First we assume 
$$\chi = \overline{\chi} \otimes {\xi}\ \  (\Leftrightarrow \overline{\chi}= \chi\otimes \overline{\xi}).$$
Let $R(x)=(r_{ij}(x))_{1\leq i,j \leq d}$ be a  unitary matrix representation  affording $\chi$. 
Then $D_{\xi}(x){}^t\!{R(x)}^{-1}=D_{\xi}(x)\overline{R(x)}$, where $D_{\xi}(x)={\rm diag}(\xi(x),\cdots,\xi(x))$ , is a unitary matrix representation affording $\xi \otimes \overline{\chi}$.
Under our assumption, Proposition \ref{Kondo}-(1) gives us
\begin{align}\label{1}
R(x)=P_{\xi,R}D_{\xi}(x)\overline{R(x)}P_{\xi,R}^{*},
\end{align}
where $P_{\xi,R}=(p_{ij})$ is a unitary matrix.
Taking the complex conjugate of (\ref{1}), we have
\begin{align}\label{2}
\overline{R(x)}=\overline{P_{\xi,R}}\overline{D_{\xi}(x)}R(x){}\overline{P_{\xi,R}^*}.\
\end{align}
From (\ref{1}) and (\ref{2}) we have
$$R(x)={P_{\xi,R}}\overline{P_{\xi,R}}R(x)\overline{P_{\xi,R}^*}P_{\xi,R}^{*}.$$
Since  $P_{\xi,R}\overline{P_{\xi,R}}$ is a scalar unitary matrix, we can put
$${P_{\xi,R}}\overline{P_{\xi,R}}=\alpha E,$$
where $|\alpha|=1$ and $E={\rm diag}(1,\cdots,1)$. 
We compute
\begin{align*}
\frac{1}{|G|}\sum_{x \in G} \overline{\xi(x)} \chi (x^2)&=\frac{1}{|G|}\sum_{x \in G}
\overline{\xi(x)} {\rm tr} R(x^2)
=\frac{1}{|G|}\sum_{x \in G}
\overline{\xi(x)} {\rm tr} R(x)R(x)\\
&=\frac{1}{|G|}\sum_{x \in G}
\overline{\xi(x)}{\rm tr} R(x)R(x^{-1})^*
=\frac{1}{|G|}\sum_{x \in G}
\left(\sum_{ij}\overline{\xi(x)}r_{ij}(x)\overline{r_{ij}(x^{-1})}\right)\\
&=\frac{1}{|G|}\sum_{ij} 
\left(\sum_{x \in G}\overline{\xi(x)}r_{ij}(x)\overline{r_{ij}(x^{-1})}\right)
=\frac{1}{|G|}\sum_{ij} 
\left(\sum_{x \in G}(\sum_{k,\ell}p_{ik}\overline{r_{k \ell }(x)}\overline{p_{j \ell}})\overline{r_{ij}(x^{-1})}\right)\\
&=\sum_{k,\ell}\sum_{ij} p_{ik}\overline{p_{j \ell}}
\left(\frac{1}{|G|}\sum_{x \in G}\overline{r_{k\ell}(x)}\overline{r_{ij}(x^{-1})}\right)
=\frac{{\rm tr}{P_{\xi,R}}\overline{P_{\xi,R}}}{d}=\alpha.
\end{align*}
We remark that $\alpha$ does not depend on choices of $R$.\\
Second
we assume 
$$\chi \not=\overline{\chi}\otimes {\xi}.$$
We compute
\begin{align*}
\frac{1}{|G|}\sum_{x \in G} \overline{\xi(x)} \chi (x^2)&=\frac{1}{|G|}\sum_{x \in G}
\overline{\xi(x)} {\rm tr} R(x^2)=\frac{1}{|G|}\sum_{x \in G}
\overline{\xi(x)} {\rm tr} R(x)R(x)\\
&=\frac{1}{|G|}\sum_{x \in G}
\overline{\xi(x)} {\rm tr} R(x)R(x^{-1})^*
=\frac{1}{|G|}\sum_{x \in G}
\left(\sum_{ij} \overline{\xi(x)}r_{ij}(x)\overline{r_{ij}(x^{-1})}\right)\\
&=\frac{1}{|G|}\sum_{ij} 
\left(\sum_{x \in G}
r_{ij}(x)\overline{r_{ij}(x^{-1})}{\xi(x^{-1})}\right)=0.
\end{align*}
\end{proof}
\begin{proposition}\label{mpf}
$(SG_{2},HG_{1},\xi\otimes \varepsilon)$ is a  Gelfand triple.\\
\end{proposition}
\begin{proof}
The  irreducible characters of $SG_{2}$ are of the following forms; 
\begin{align*}
\begin{cases}
T^{\chi}_{1}(g,h;\sigma)=\chi(g)\chi(h)\delta_{1\sigma}+\chi(gh)\delta_{(12)\sigma}, & \\ 
T^{\chi}_{\rm sgn}(g,h;\sigma)={\rm sgn}(\sigma)(\chi(g)\chi(h)\delta_{1\sigma}+\chi(gh)\delta_{(12)\sigma}),&\\
U^{\chi,\eta}(g,h;\sigma)=\delta_{1,\sigma}(\chi(g)\eta(h)+\chi(h)\eta(g)),&(\chi\not=\eta),
,
\end{cases}
\end{align*}
where $\chi$ and $\eta$ are irreducible characters of $G$.
We remark that $U^{\chi,\eta}=U^{\eta,\chi}$.
For $\varepsilon_{1}=1$ or ${\rm sgn}$,
we compute
\begin{align*}
\langle T_{\varepsilon_1}^{\chi},\xi\otimes \varepsilon \rangle_{HG_{1}}&=
\frac{1}{2|G|}\sum_{g\in G}\chi(g)\chi(g)\overline{\xi(g)}
+\varepsilon_{1}{(12)}\varepsilon{(12)}
\frac{1}{2|G|}\sum_{g\in G}\chi(g^2)\overline{\xi(g)}\\
&=\frac{1}{2}\langle \chi,\bar{\chi}\otimes {\xi}\rangle_{G}
+\varepsilon_{1}{(12)}\varepsilon{(12)}
\frac{1}{2|G|}\sum_{g\in G}\chi(g^2)\overline{\xi(g)}\\
&=\begin{cases}
\frac{1}{2}
+\varepsilon_{1}{(12)}\varepsilon{(12)}
\frac{\alpha}{2}&(\chi = \bar{\chi} \otimes {\xi})\\
0&(\chi \not= \bar{\chi} \otimes {\xi}).
\end{cases}
\end{align*}
Since $|\alpha|=1$ and $\langle T_{\varepsilon_1}^{\chi},\xi\otimes \varepsilon \rangle_{HG_{1}}
$ is an integer, we have $\alpha=\pm1$ and
\begin{align*}
\langle T_{\varepsilon_1}^{\chi},\xi\otimes \varepsilon \rangle_{HG_{1}}
&=\begin{cases}
1&(\chi = \bar{\chi} \otimes {\xi}\ \ and\ \ \varepsilon_{1}(12)=\alpha \varepsilon(12)),\\
0&(otherwise).
\end{cases}
\end{align*}
\begin{align*}
\langle (\chi,\eta;1),\xi\otimes \varepsilon_1 \rangle_{HG_{1}}&=
\frac{1}{2|G|}\sum_{g\in G}2\chi(g)\eta(g)\overline{\xi(g)}
=\langle \chi,\bar{\eta}\otimes {\xi}\rangle_{G}\\
&=\begin{cases}
1&(\chi = \bar{\eta} \otimes {\xi})\\
0&(\chi \not= \bar{\eta} \otimes {\xi}).
\end{cases}
\end{align*}
Therefore $\xi \otimes \varepsilon \uparrow$ is multiplicity free.
Simultaneously we have all irreducible components.
\end{proof}
Also we have
\begin{proposition}\label{2irrdcm}
\begin{enumerate}
\item
\begin{align*}
(\xi \otimes 1) \uparrow_{HG_{1}}^{SG_{2}}=
\bigoplus_{\substack{\chi=\xi\otimes \overline{\chi}\\ \nu_2^{\xi}(\chi)=1}} 
 S^{(2)}(\chi)
 \oplus 
 \bigoplus_{\substack{\chi=\xi\otimes \overline{\chi}\\ \nu_2^{\xi}(\chi)=-1}}
  S^{(1^2)}(\chi)
  \oplus
 \bigoplus_{\substack{\chi \not=\xi\otimes \overline{\chi}\\ \nu_2^{\xi}(\chi)=0}}
S^{(1)}(\chi)\otimes S^{(1)}(\xi \otimes \overline{\chi})\uparrow_{SG_{1}\times SG_{1}}^{SG_{2}}
\end{align*}
\item
\begin{align*}
(\xi \otimes {\rm {sgn}})\uparrow_{HG_{1}}^{SG_{2}}=
\bigoplus_{\substack{\chi=\xi\otimes \overline{\chi}\\ \nu_2^{\xi}(\chi)=1}} 
 S^{(1^2)}(\chi)
 \oplus 
 \bigoplus_{\substack{\chi=\xi\otimes \overline{\chi}\\ \nu_2^{\xi}(\chi)=-1}}
 S^{(2)}(\chi)
  \oplus
 \bigoplus_{\substack{\chi \not=\xi\otimes \overline{\chi}\\ \nu_2^{\xi}(\chi)=0}}
S^{(1)}(\chi)\otimes S^{(1)}(\xi \otimes \overline{\chi})\uparrow_{SG_{1}\times SG_{1}}^{SG_{2}}
\end{align*}
\end{enumerate}
\end{proposition}

From the proof of Proposition \ref{mpf}, we have a twisted version of
the Frobenius-Schur's theorem (See also  \cite[Theorem 9]{bumpfs}).
\begin{theorem}\label{modFS}
Let $\chi$ be an irreducible character of $G$.
Then we have
\begin{align*}
\nu_{2}^{\xi}(\chi)=\frac{1}{|G|}\sum_{x \in G} \overline{\xi(x)} \chi (x^2)=
\begin{cases}
\pm 1 & (\chi = \overline{\chi} \otimes {\xi})\\
0 & (\chi \not= \overline{\chi} \otimes {\xi}).
\end{cases}
\end{align*}
\end{theorem}
\
Next we consider a basis of ${\mathcal H}^{\xi,\varepsilon}(SG_{2},HG_{1})$ associated with the double coset. 
Fix a complete representatives $\{g_{R}\mid  g_{R}\in R\}$ of  $G_{**}$. 
For $(g_{1},g_{2}:\sigma) \in SG_{2}$, we define a element of ${\mathcal H}^{\xi,\varepsilon}(SG_{2},HG_{1})$ by
$$K_{(g_1,g_2;\sigma)}=\sum_{x,y \in HG_{1} }\varepsilon(xy)\xi(xy)x(g_1,g_2;\sigma) y.$$
Clearly $\{K_{x}| x \in SG_{2} \}$ spans ${\mathcal H}^{\xi,\varepsilon}(SG_{2},HG_{1})$.
\begin{proposition}\label{zeropt}
$K_{(1,g;\sigma)}=0$ if and only if $\xi(g)=-1\ {\rm and}\ C_{g}=C_{g^{-1}}$.
\end{proposition}
\begin{proof}
Fix an element $g \in G$. We consider elements $a,b \in G$ and $\tau,\theta \in S_{2}$ satisfying 
$$(a,a;\tau)(1,g:\sigma)(b,b;\theta)=(1,g;\sigma).$$ 
Since $\sigma=1$ or $(12)$, we have $\tau=\theta$.

First we consider the case of $\tau=\theta=1$.
Then we have $ab=1$ and $aga^{-1}=g$.

Second we consider the case of $\tau=\theta=(12)$.
Then we have $ab=g$ and $aga^{-1}=g^{-1}$.
Therefore the conjugacy class of $g$ is a real and $\xi(g)=\pm1$. 
When  $\xi(g)=-1$
, a computation
$$K_{(1,g;\sigma)}=K_{(a,a;(12))(1,g:\sigma)(b,b;(12))}=\xi(ab)K_{(1,g;\sigma)}=\xi(g)K_{(1,g;\sigma)}
=-K_{(1,g;\sigma)},$$
gives us  $K_{(1,g;\sigma)}=0$.
Let ${\bf Co}(K_{(1,g:\sigma)})$ be the coefficient of $(1,g:\sigma)$ in $K_{(1,g;\sigma)}$.
The same computation gives us 
$
{\bf Co}(K_{(1,g:\sigma)})=
\begin{cases}
0& (\xi(g)=-1\ {\rm and}\ C_{g}=C_{g^{-1}}) \\
2
\zeta_{g}
& (\xi(g)=1\ {\rm and}\  C_{g}=C_{g^{-1}})\\
\zeta_{g}
& (C_{g} \not=C_{g^{-1}}).
\end{cases}
$
\end{proof}
Clearly $\{K_{g_{R}}\ |\ R\in G_{**}\ {\rm such \ that}\ \ K_{g}\not=0 \}$ are linearly independent.
Also Proposition \ref{zeropt} gives us 
\begin{proposition}\label{dc2bases}
A set
$$\{K_{g_{R}}\ |\ R\in G_{**}.\ {\rm If\ {\it R}\ is\ real,\ then  }\ \ \xi(g_{R}) \not=-1. \}$$ 
gives a basis of ${\mathcal H}^{\xi,\varepsilon}(SG_{2},HG_{1})$.
\end{proposition}
Combining  Proposition \ref{mpf} and Proposition \ref{dc2bases},
we have 
\begin{proposition}\label{number}
We denote by $\xi(C)$ the value of $\xi$ on a conjugacy class $C$. Then we have
\begin{align*}
|\{\chi|\chi=\overline{\chi}\otimes {\xi}\}|+\frac{1}{2}|\{\chi|\chi\not=\overline{\chi}\otimes {\xi}\}|
&=|G_{**}|-|\{C|C =C^{-1},\xi(C)=-1\}|.
\end{align*}
\end{proposition}

Put 
$$n_{\xi}=|\{C|C =C^{-1},\xi(C)=-1\},  n_{R,\xi}=|\{C|C =C^{-1},\xi(C)=1\}|,n_{C}=\{C|C\not=C^{-1}\},$$ 
$$n_{R}=n_{**}-\frac{1}{2}n_{C}, n^{R,\xi}=|\{\chi|\chi=\overline{\chi}\otimes {\xi}\}|\ {\rm{ and}}\ n^{C,\xi}=|\{\chi|\chi\not=\overline{\chi}\otimes {\xi}\}|.$$
Corollary \ref{number} gives us
\begin{align}\label{eqA}
n^{R,\xi}+\frac{1}{2}n^{C,\xi}=n_{**}-n_{\xi}.
\end{align}
The following is just the relation between the number of the irreducible representation and
the number of the conjugacy classes.  
\begin{align}\label{eqB}
n^{R,\xi}+n^{C,\xi}=n_{**}+\frac{1}{2}n_{C}.
\end{align}
From (\ref{eqA}) and (\ref{eqB}), we have
\begin{theorem}\label{countinglemma}
$$\frac{1}{2}n_{C}+n_{\xi}=\frac{1}{2}n^{C,\xi}$$
and
$$n_{R}-2n_{\xi}=n^{R,\xi}.$$
\end{theorem}
This proposition is a generalization of the fact $n_{C}=n^{C,1}$ and $n_{R}=n^{R,1}$.
\begin{example}
The character table of  $G={\rm GL}_{2}({\mathbb F}_3)$ is
$$
\left[\begin{array}{c|cccccccc|cc}
G_{*}&C_{1}&C_{2}&C_{3	}&C_{4}&C_{5}&C_{6}&C_{7}&C_{8}&&\\
\hline
G_{**}&R_1&R_2&R_3&{R_{4}}&{R_{5}}&{R_{6}}&{R_{7}}&{R_{7}}&\nu_{2}^{1}(\chi)&\nu_{2}^{\xi}(\chi)\\
\hline
\chi_1 & 1 & 1 & 1 & 1 & 1 & 1 & 1 & 1 &1&0\\
{\chi_2} & 1 & 1 & 1 & 1 & 1 & {\bf -1} & -1 & -1 &1&0\\
\chi_3 & 2 & 2 & 2 & -1 & -1 & 0 & 0 & 0 &1&-1\\
\chi_4 & 3 & 3 & -1 & 0 & 0 & 1 & -1 & -1 &1&0\\
\chi_5 & 3 & 3 & -1 & 0 & 0 & -1 & 1 & 1 &1&0\\
\chi_6 & 2 & -2 & 0 & -1 & 1 & 0 & \sqrt{2}i & -\sqrt{2}i &0&-1\\
\chi_7& 2 & -2 & 0 & -1 & 1 & 0 & -\sqrt{2}i & \sqrt{2}i &0&-1\\
\chi_8 & 4 & -4 & 0 & 1 & -1 & 0 & 0 & 0&1&-1
 \end{array}\right].
$$
We set $\xi=\chi_{2}$ and  have
$$  \chi_j =\xi \otimes \overline{\chi_j}\ (j=3,6,7,8),\ \chi_1=\xi \otimes \overline{\chi_2}\ and\  \ \chi_4 =\xi \otimes \overline{\chi_5}.$$
In this case, we have $n_{\xi}=1=|\{C_{6}\}|$, $n_{C}=2$ and $n^{C,\xi}=4=|\{\chi_{i};i=1,2,4,5\}|$.
\end{example}

%
\section{Irreducible Decomposition of Induced Representations}
\subsection{linear representation of $HG_{n}$}
The hyperoctahedral group $H_{n}$ has exactly four linear characters; 
$${\mathcal L}(H_{n})=\{1,\delta,\iota,\delta \otimes \iota\},$$
which are defined by
$$
\begin{cases}
\delta ((2i\!-\!1,2i))=-1\\
 \delta(\sigma)=1
\end{cases}
\ {\rm and}\ \ \ 
\begin{cases}
\iota ((2i\!-\!1,2i))=1,\\
\iota(\sigma)={\rm sgn}(\phi_{n}^{-1}(\sigma))
\end{cases}$$
for $1 \leq i \leq n$ and $\sigma \in \phi(S_{n})$.
We remark that $\delta$ is obtained by restricting the sign representation of $S_{2n}$
to $H_{n}$.
The following are known as the Littlewood's formula.
\begin{proposition}\label{littlewood}
\begin{enumerate}
\item
$1\uparrow_{H_{n}}^{S_{2n}}=\bigoplus_{\lambda \in P_n}S^{2\lambda}$.
\item
$\delta\uparrow_{H_{n}}^{S_{2n}}=\bigoplus_{\lambda \in P_n}S^{(2\lambda)'}$.
\item
$ \iota\uparrow_{H_{n}}^{S_{2n}}=\bigoplus_{\lambda \in SP_n}S^{D(\lambda)}$.
\item
$\delta \otimes \iota\uparrow_{H_{n}}^{S_{2n}}=\bigoplus_{\lambda \in SP_n}S^{D(\lambda)'}$.
\end{enumerate}
\end{proposition}
Let $\pi \in {\mathcal L}(H_{n})$ and $\xi\in {\mathcal L}(G)$.
We define  $\Theta_{\xi,\pi}\in {\mathcal L}(HG_{n})$ by
$$\Theta_{\xi,\pi}(g_{1},g_{1},\cdots,g_{n},g_{n};\sigma)=\xi(g_{1}g_{2}\cdots g_{n})\pi(\sigma)$$
for  $(g_{1},g_{1},\cdots,g_{n},g_{n};\sigma) \in HG_{n}$.

\begin{proposition}
${\mathcal L}(HG_n)=\{\Theta_{\xi,\pi}\mid \xi \in {\mathcal L}(G), \pi \in {\mathcal L}(H_{n})\}$.
\end{proposition}
\begin{proof}
We remark the following isomorphism,
$HG_{n} \cong (G\times {\mathbb Z}/2{\mathbb Z}) \wr S_{n}.$
Therefore  $|{\mathcal L}(HG_{n})|= G/[G,G] \times 2 \times 2$.
It is clear that $(\xi,\pi) \not = (\xi',\pi')$, if $\xi \not= \xi'$ or $\pi \not= \pi'$.
Therefore $|\{\Theta_{\xi,\pi}\}|=4\times G/[G,G].$
\end{proof}
%
\subsection{basis associated with double coset}
\begin{definition}
We define two subsets of $P_{**}(n)$ by
$$P_{**}^{\xi,+}(n)=\left\{ (\rho(R)|R \in G_{**}); \rho(R)\ is\ 
\begin{cases}
=\emptyset,\ &\ (R\ is\ real\ and\ \xi \equiv-1\ on\ R),\\
\in P,\ &\ (otherwise).
\end{cases}\right\}$$
and
$$
P_{**}^{\xi,-}(n)=\left\{ 
 (\rho(R)|R \in G_{**});
\rho(R) \in
\begin{cases}
OP\ &\ (R\ is\ real\ and\ \xi\equiv1\ on\ R),\\
EP,\ &\ (R\ is\ real\ and\ \xi\equiv -1\ on\ R),\\
P,\ &\ (otherwise).
\end{cases}
\right\}.$$
\end{definition}
Put
$$x_{m}=\prod_{k=1}^{m}(k,2m\!\!-\!\!k\!\!-\!\!1)(2m\!\!-\!\!1,2m),$$
and
$$y_{m}=(2m\!\!-\!\!1,2m\!\!-\!\!3,\cdots,3,1)(2m,2m\!\!-\!\!2,\cdots,4,2).$$
Since  $x_{m}$ and $y_{m}$ are commutative with $\prod_{i=1}^{n}(2i-1,2i)$ , we have $x_{m}, y_{m} \in H_{n}$.

\begin{proposition}
Put $\sigma=(1,2,\cdots,2m)$ and $g \in G$. 
We assume that there exists $z \in G$ such that $zgz^{-1}= g^{-1}$. Then we have
$$(\underbrace{g^{-1}z,\cdots,g^{-1}z}_{2m-2},z,z:x_{m})(1,\cdots,1,g:\sigma)
(\underbrace{z^{-1}g,\cdots,z^{-1}g}_{2m}:y_{m}x_{m})=(1,\cdots,1,g:\sigma).$$
\end{proposition}
\begin{proof}
Since $\sigma^2 =(1,3,5,\cdots,2m-1)(2,4,6,\cdots,2m)=y_{m}^{-1}$
and
$x_{m}\sigma^{-1}x_{m}=\sigma$,
 we have
$x_{m}\sigma y_{m}x_{m}=\sigma$.
\end{proof}
\begin{proposition}\label{-1}
We have
$$(g^{-1}z,\cdots,g^{-1}z,z,z:x_{m})(z^{-1}g,\cdots,z^{-1}g:y_{m}x_{m})=(1,\cdots,1,g,g:\tau_{m}),$$
where $\tau_{m}=(246\cdots 2m)(135\cdots2m-1),$
and 
$$\Theta_{\xi,\varepsilon}((g^{-1}z,\cdots,g^{-1}z,z,z:x_{m})(z^{-1}g,\cdots,z^{-1}g:y_{m}x_{m}))=
\begin{cases}
\xi(g) & (\varepsilon=1,\delta),\\
(-1)^{m+1}\xi(g) &  (\varepsilon=\iota,\delta \otimes \iota).
\end{cases}$$
\end{proposition}
\begin{proof}
A direct computation gives us
$x_{m} y_{m}x_{m}=\tau_{m} \in H_{n}$.
Since $\phi^{-1}_{n}(\tau_{m})=(123\cdots m)$, we have $\varepsilon(\tau_{m})=(-1)^{m+1}$.
\end{proof}
\begin{example}
For $g \in G$, suppose there exists $z \in G$ such that $zgz^{-1}=g^{-1}$.
Put $\sigma=(123456)$.
\begin{align*}
(g^{-1}\!z,g^{-1}\!z,g^{-1}\!z,g^{-1}\!z,z,z;x_3)(1,1,1,&1,1,g;\sigma)
(z^{-1}\!\!g,z^{-1}\!g,z^{-1}g,z^{-1}\!g,z^{-1}\!g,z^{-1}\!g:y_{3}x_3)\\
&=(1,1,1,1,1,g;\sigma)
\end{align*}
{$$ (g^{-1}\!z,g^{-1}\!z,g^{-1}\!z,g^{-1}\!z,z,z;x_{3})
(z^{-1}\!g,z^{-1}\!g,z^{-1}\!g,z^{-1}\!g,z^{-1}\!g,z^{-1}\!g:y_{3}x_{3})=(1,1,1,1,g,g:\tau_{3})$$}
\end{example}
\begin{proposition}\label{doublecoset1}
Let $x(\underline{\rho})$ be an element as in Theorem \ref{doublecoset}
for $\underline{\rho}=(\rho(R)|R \in G_{**}) \in P_{**}(n)$,
\begin{enumerate}
\item For $\varepsilon=1$ or $\delta$, we have
$$e_{n}^{\xi,\varepsilon} x(\underline{\rho}) e_{n}^{\xi,\varepsilon}
 \not=0 
 \Rightarrow \underline{\rho} \in P_{**}^{\xi,+}(n).$$
\item
For $\varepsilon=\iota$ or $\delta \otimes \iota$, we have
$$e_{n}^{\xi,\varepsilon} x(\underline{\rho}) e_{n}^{\xi,\varepsilon}
 \not=0
  \Rightarrow \underline{\rho} \in P_{**}^{\xi,-}(n).$$
\end{enumerate}
\end{proposition}
\begin{proof}
We assume $\underline{\rho} \not \in P_{**}^{\xi,\pm}(n)$.
Then we can choose $y,z \in HG_{n}$ which satisfy
$yx(\underline{\rho})z=x(\underline{\rho})$
and
 $\Theta_{\xi,\varepsilon}(yz)=-1$ 
from Proposition \ref{-1}.
Since
$e_{n}^{\xi,\varepsilon}y x(\underline{\rho}) ze_{n}^{\xi,\varepsilon}
=\Theta_{\xi,\varepsilon}(yz)e_{n}^{\xi,\varepsilon} x(\underline{\rho}) e_{n}^{\xi,\varepsilon}
=-e_{n}^{\xi,\varepsilon} x(\underline{\rho}) e_{n}^{\xi,\varepsilon}$,
 we have $e_{n}^{\xi,\varepsilon}y x(\underline{\rho}) ze_{n}^{\xi,\varepsilon}=0$.
\end{proof}
\subsection{permutation representations}
Fix $\xi \in {\mathcal L}(G)$.
\begin{definition}
Let $X'=\{\lambda' \mid \lambda \in X \}$ for a subset $X$ of $P$.
We define the following four subsets of $P^{*}(n)=\{\underline{\lambda}=(\lambda(\chi)|\chi \in G^*)\ \mid\ |\underline{\lambda}|=2n\}$.
\begin{enumerate}
\item
$\displaystyle{P^{**}_{\xi,1}(n)=
\left\{
(\lambda(\chi)|\chi \in G^*)
\mid 
\lambda(\chi)\ is\ 
{\small
\begin{cases}
\in E,\ &\ (\nu_{2}^{\xi}(\chi)=1),\\
\in E',\ &\ (\nu_{2}^{\xi}(\chi)=-1),\\
=\lambda({\xi}
\otimes \overline{\chi}) \in P, & \ (\nu_{2}^{\xi}(\chi)=0).
\end{cases}
}
\right\}.}$
\item
$\displaystyle{
P^{**}_{\xi,\delta}(n)=
\left\{
(\lambda(\chi)|\chi \in G^*)
\mid 
\lambda(\chi)\ is\ 
{\small \begin{cases}
\in E,\ &\ (\nu_{2}^{\xi}(\chi)=-1),\\
\in E',\ &\ (\nu_{2}^{\xi}(\chi)=1),\\
=\lambda({\xi}
\otimes \overline{\chi})\in P, & \ (\nu_{2}^{\xi}(\chi)=0).
\end{cases}}
\right\}.
}$
\end{enumerate}
Put $DSP=\{D(\lambda) \mid \lambda \in SP\}$.
\begin{enumerate}
\item[(3)]
$\displaystyle{
P^{**}_{\xi,\iota}(n)=
\left\{
(\lambda(\chi)|\chi \in G^*)\
\mid 
\lambda(\chi)\ is\ 
{\small \begin{cases}
 \in DSP,\ &\ (\nu_{2}^{\xi}(\chi)=1),\\
\in DSP',\ &\ (\nu_{2}^{\xi}(\chi)=-1),\\
= \lambda({\xi}
\otimes \overline{\chi})'\in P, & \ (  
 \nu_{2}^{\xi}(\chi)=0).
\end{cases}}
\right\}.
}$
\item[(4)]
$\displaystyle{
P^{**}_{\xi,\delta \otimes \iota}(n)=
\left\{
(\lambda(\chi)|\chi \in G^*)
 \mid 
\lambda(\chi)\ is\ 
{\small \begin{cases}
\in DSP,\ &\ (\nu_{2}^{\xi}(\chi)=-1),\\
\in DSP',\ &\ (\nu_{2}^{\xi}(\chi)=1),\\
=\lambda({\xi}
\otimes \overline{\chi})'\in P, & \ (  
 \nu_{2}^{\xi}(\chi)=0).
\end{cases}}
\right\}.
}$
\end{enumerate}
\end{definition}
\begin{proposition}\label{equals}
\begin{enumerate}
\item
$|P^{**}_{\xi,1}(n)|=|P^{**}_{\xi,\delta}(n)|=|P_{**}^{\xi,+}(n)|.$
\item
$|P^{**}_{\xi,\iota}(n)|=|P^{**}_{\xi,\delta \otimes \iota}(n)|=|P_{**}^{\xi,-}(n)|.$
\end{enumerate}
\end{proposition}
\begin{proof}
It is easy to obtain the first claim by using Theorem \ref{countinglemma}.
The second claim
 can be obtained by using Theorem \ref{countinglemma} and  two bijections:
$$SP_{n}\rightarrow OP_{n} \ {\rm and}\ P_{n} \rightarrow \bigcup_{n_{0}+n_{1}=n} EP_{n_{0}}\cup OP_{n_{1}}.$$
\end{proof}
\begin{example}
We consider the character table of the quaternion group $Q_{8}$. Put $\xi=\chi_{2}$.
$$
\left[\begin{array}{c|ccccc|cc}
   & C_1 & C_2 & C_3 & C_4 & C_5 &  &  \\
  & R_1 & R_2 & R_3 & R_4 & R_5 & \nu_2^{1} & \nu_{2}^{\xi} \\
\hline
{\chi_{1}} & 1 & 1 & 1 & 1 & 1 & 1 & 0\\
{\chi_{2}}& 1 & {-1} & 1 & 1 & {-1} & 1 & 0 \\
{\chi_{3}} & 1 & 1 & 1 & -1 & -1 & 1 & 0 \\
{\chi_{4}} & 1 & -1 & 1 & -1 & 1 & 1 & 0 \\
\chi_{5} & 2 & 0 & -2 & 0 & 0 & -1 & 1
\end{array}\right].$$
We have the following bijections, 
$$P_{\xi,\varepsilon}^{**}(n)=\{\underline{\lambda}=(\lambda,\lambda,\mu,\mu,D(\nu))|\lambda,\mu \in {P},
\nu \in SP \}
\longrightarrow
\bigcup_{n_{0}+n_{1}+n_{2}=n}\!\!\!\!\!\!P_{n_{0}}\cup P_{n_{1}}\cup SP_{n_{2}}
$$
and
\begin{align*}
&P^{\xi,\varepsilon}_{**}(n)=\{(\lambda^{1}, \lambda^{2},\lambda^{3},\lambda^{4},\lambda^{5})|
\lambda^{1},\lambda^{3},\lambda^{5} \in P, \lambda^{2},\lambda^{4} \in EP\}\\
&\longrightarrow \!\!\!\!\!\!\!\!\!\!\!\!
\bigcup_{n_{1}+n_{2}+n_{3}+n_{4}+n_{5}=n}\!\!\!\!\!\!\!\!\!\!\!\!OP_{n_{0}}\cup EP_{n_{1}}\cup OP_{n_{2}}\cup OP_{n_{3}}\cup EP_{n_{4}}
\longrightarrow
\bigcup_{n_{0}+n_{1}+n_{2}=n}\!\!\!\!\!\!
P_{n_{0}}\cup OP_{n_{1}} \cup P_{n_{2}}.
\end{align*}
By using $SP_{n}\twoheadrightarrow OP_{n}$, we have
$|P_{\xi,\varepsilon}^{**}(n)|=|P^{\xi,\varepsilon}_{**}(n)|$ .
\end{example}
Let $\chi$ be an irreducible character of $G$.
For  $\lambda,\mu\in P_{n}$ and $\xi \in {\mathcal L}(G)$, 
we denote by $S^{\lambda,\mu}(\chi,\xi)$ a irreducible representation 
$S^{\lambda,\mu}(\chi,\xi)=S^{\lambda}(\chi) \otimes S^{\mu}(\overline{\chi}\otimes \xi)
\uparrow_{SG_{n}\times SG_{n}}^{SG_{2n}}$ of $SG_{2n}$.
\begin{proposition}\label{smallinv}
Let $\xi \in {\mathcal L}(G)$. Put $x=(g_{1},g_{1},g_{2},g_{2},\cdots,g_{n},g_{n};\sigma) \in HG_{n}$.
\begin{enumerate}
\item
$S^{2\lambda}(\chi)$ is an irreducible component of 
$\begin{cases}
\Theta_{\xi,1}(x)\uparrow_{HG_{n}}^{SG_{2n}} & (\nu_{2}^{\xi}(\chi)=1) \\
\Theta_{\xi,\delta}(x)\uparrow_{HG_{n}}^{SG_{2n}}  & (\nu_{2}^{\xi}(\chi)=-1). 
\end{cases}$
\item
$S^{(2\lambda)'}(\chi)$ 
 is an irreducible component of 
$\begin{cases}
\Theta_{\xi,1}(x)\uparrow_{HG_{n}}^{SG_{2n}} & (\nu_{2}^{\xi}(\chi)=-1) \\
\Theta_{\xi,\delta}(x)\uparrow_{HG_{n}}^{SG_{2n}}  & (\nu_{2}^{\xi}(\chi)=1). 
\end{cases}$
%
\item
$S^{D(\lambda)}(\chi)$ 
is an irreducible component of 
$\begin{cases}
\Theta_{\xi,\iota}(x)\uparrow_{HG_{n}}^{SG_{2n}} & (\nu_{2}^{\xi}(\chi)=1) \\
\Theta_{\xi,\delta \otimes \iota}(x)\uparrow_{HG_{n}}^{SG_{2n}}  & (\nu_{2}^{\xi}(\chi)=-1). 
\end{cases}$
\item
$S^{D(\lambda)'}(\chi)$ 
is an irreducible component of 
$\begin{cases}
\Theta_{\xi,\iota}(x)\uparrow_{HG_{n}}^{SG_{2n}} & (\nu_{2}^{\xi}(\chi)=-1) \\
\Theta_{\xi,\delta \otimes \iota}(x)\uparrow_{HG_{n}}^{SG_{2n}}  & (\nu_{2}^{\xi}(\chi)=1). 
\end{cases}$

\item
If $\nu_{2}^{\xi}(\chi)=0$, then  
$S^{\lambda,\lambda}(\chi,\xi)$ 
is an irreducible component of $\Theta_{\xi,\pi}(x)\uparrow_{HG_{n}}^{SG_{2n}}$  for $\pi=1$ or $\delta$.
\item
If $\nu_{2}^{\xi}(\chi)=0$, then  
$S^{\lambda,\lambda'}(\chi,\xi)$ is an irreducible component of $\Theta_{\xi,\pi}(x)\uparrow_{HG_{n}}^{SG_{2n}}$  
for $\pi=\iota$ or $\delta \otimes \iota$.
\end{enumerate}
\end{proposition}
\begin{proof}
We write $\sigma \in H_{n}$ in the form of $\sigma=(12)^{\epsilon_{1}}(34)^{\epsilon_{2}}
\cdots (2n-1\ 2n)^{\epsilon_{n}} \tau $,
where $\epsilon_{i} \in \{0,1\}$ and $\tau \in \phi^{-1}(S_{n})$.
Firstly we consider (1). If $\nu_{2}^{\xi}(\chi)=1$ (resp. $-1$), then $V_{\chi} \otimes V_{\chi}$ has an
 element $f_{\chi}$ (resp. $g_{\chi}$) such that $(g,g:(12))f_{\chi}=\xi(g)f_{\chi}$ 
 (resp. $(g,g:(12))g_{\chi}=-\xi(g)g_{\chi}$)
 from Proposition \ref{2irrdcm}.
From proposition \ref{littlewood}, $S^{2 \lambda}$
 has an element $u_{2\lambda}$  such that 
$\sigma u_{\lambda}=u_{\lambda}$.
We consider $f_{\chi}^{\otimes n} \otimes u_{\lambda} \in S^{2\lambda}(\chi)$ and
compute
\begin{align*}
x f_{\chi}^{\otimes n} \otimes u_{\lambda} &=
(g_{1},g_{1}:(12)^{\varepsilon_{1}})f_{\chi}
\otimes
\cdots
\otimes
(g_{n},g_{n}:(2n-1\ 2n)^{\varepsilon_{n}})f_{\chi}
\otimes
\sigma u_{\lambda}\\
&=\xi(g_{1})\xi(g_{2})\cdots \xi(g_{n}) f_{\chi}^{\otimes n}\otimes  u_{\lambda}
=\xi(g_{1}g_{2}\cdots g_{n}) f_{\chi}^{\otimes n}\otimes  u_{\lambda}.
\end{align*}
Also we consider $g_{\chi}^{\otimes n} \otimes v_{\lambda} \in S^{2\lambda}(\chi)$ and
compute
\begin{align*}
x {g_{\chi}}^{\otimes n} \otimes v_{\lambda} &=
(g_{1},g_{1}:(12)^{\varepsilon_{1}}){g_{\chi}}
\otimes
\cdots
\otimes
(g_{n},g_{n}:(2n-1\ 2n)^{\varepsilon_{n}}){g_{\chi}}
\otimes
\sigma u_{\lambda}\\
&=(-1)^{\varepsilon_{1}+\cdots+\varepsilon_{n}}\xi(g_{1})\xi(g_{2})\cdots \xi(g_{n})
g_{\chi}^{\otimes n}\otimes  u_{\lambda}
=\delta(\sigma)\xi(g_{1}g_{2}\cdots g_{n})
{g_{\chi}}^{\otimes n}\otimes  u_{\lambda}.
\end{align*}
(2),(3) and (4) are obtained by the same way.
Second
we consider the case of $\nu_{2}^{\xi}(\chi)=0$. 
%
%
%
Let $\Delta G$ be the diagonal subgroup of $G \times G$.
We consider the irreducible representations $\chi_1 \otimes \chi_2$ of $G \times G$.
Then an easy computation of characters gives us the following formula
of the intertwining number; 
$\langle \chi_{1}\otimes {\chi_2},\xi \rangle_{\Delta G}=\delta_{\chi_{2},\overline{\chi_{1}}\otimes \xi}$
(We will consider more detailed discussion of this fact in Proposition \ref{ccc}).  
We remark an isomorphism $S^{\lambda}(\overline{\chi} \otimes \xi) \cong S^{\lambda}(\overline{\chi})\otimes S^{(n)}(\xi)$.
Therefore $S^{\lambda}(\chi)\otimes S^{\lambda}(\overline{\chi} \otimes \xi)$ have an element 
such that $y k_{\lambda}=\xi(x_{1}x_{2}\cdots x_{n}) k_{\lambda}$ for $y=(x_{1},x_{1},\cdots,x_{n},x_{n}: \tau)\in \Delta SG_{n}$.
 Put $t_{+}=\prod_{i=1}^{n}(1+(2i-1\ 2i))$ and
  $t_{-}=\prod_{i=1}^{n}(1-(2i-1\ 2i))$. 
  Since $\langle (2i-1\ 2i)| 1 \leq i \leq n\rangle \cap \phi^{-1}(S_{n})
=
\{1\}$, the elements $t_{+} \otimes k_{\lambda}, t_{-} \otimes k_{\lambda}\in S^{\lambda,\lambda}(\chi,\xi)$ are nonzero and satisfy$$
\begin{cases}
xt_{+} \otimes k_{\lambda}=t_{+} \otimes y k_{\lambda}=\xi(x_{1}x_{2}\cdots x_{n})t_{+} \otimes k_{\lambda},\\
xt_{-} \otimes k_{\lambda}=\delta(\sigma)t_{-} \otimes y k_{\lambda}=\delta(\sigma)\xi(x_{1}x_{2}\cdots x_{n})t_{-} \otimes k_{\lambda}.
\end{cases}$$
Therefore we obtain (5).
By  Remarking an isomorphism  $S^{\lambda'}(\overline{\chi} \otimes \xi) 
\cong S^{\lambda}(\overline{\chi})\otimes S^{(1^n)}(\xi)$,
 (6) can be obtained by the same way of the case of (5).
\end{proof}

\begin{theorem}\label{decomposition}
\begin{enumerate}
\item
Let $\pi\in {\mathcal L}(H_{n})$ and $\xi \in G^*$. We have
$$\Theta_{\xi,\pi}\uparrow_{HG_{n}}^{SG_{2n}}=\bigoplus_{\underline{\lambda}\in P^{**}_{\xi,\pi}}S(\underline{\lambda}).$$
In particular, the triplet $(SG_{2n},HG_{n},\Theta_{\xi,\pi})$ is a Gelfand triple.
\item
Let $x(\underline{\rho})$ be an element as in Theorem \ref{doublecoset}
for $\underline{\rho}=(\rho(R)|R \in G_{**}) \in P^{1,+}_{**}(n)$,
\begin{enumerate}
\item For $\varepsilon=1$ or $\delta$, we have
$$e_{n}^{\xi,\varepsilon} x(\underline{\rho}) e_{n}^{\xi,\varepsilon}
 \not=0 
 \Leftrightarrow \underline{\rho} \in P_{**}^{\xi,+}(n).$$
\item
For $\varepsilon=\iota$ or $\delta \otimes \iota$, we have
$$e_{n}^{\xi,\varepsilon} x(\underline{\rho}) e_{n}^{\xi,\varepsilon}
 \not=0
  \Leftrightarrow \underline{\rho} \in P_{**}^{\xi,-}(n).$$
\end{enumerate}
\end{enumerate}
\end{theorem}
\begin{proof}
Set $\pi \in {\mathcal L}(H_{n})$,
  $\underline{\lambda}=(\lambda(\chi)\ |\  \chi \in G^{*} )\in P^{**}_{\xi,\pi}$,
   $n(\underline{\lambda})=(|\lambda(\chi)|\ |\  \chi \in G^{*} )$ and
$SG(n(\underline{\lambda}))=\prod_{\chi \in G^{*}}SG_{|\lambda(\chi)|}$.
We define a subgroup of $SG_{2n}$ by
$$\widetilde{SG}(n(\underline{\lambda}))
=\prod_{\substack{\chi \in G^{**}_{\xi}\\
 \nu_{2}^{\xi}(\chi)=\pm1}}
 SG_{|\lambda(\chi)|}
\times
\prod_{\substack{\chi \in G^{**}_{\xi}\\
 \nu_{2}^{\xi}(\chi)=0}}
 SG_{2|\lambda(\chi)|}.$$
 We remark that $\widetilde{SG}(n(\underline{\lambda}))$ contains 
$SG(n(\underline{\lambda}))$ as its subgroup. 
From the transitivity of the induction, we have  
$$S(\underline{\lambda})= \bigotimes_{\chi \in G^*} S^{\lambda(\chi)}(\chi)
 \uparrow_{SG(\underline{n})}^{SG_{n}}
 =
\bigotimes_{\substack{\chi \in G^{*}\\
 \nu_{2}^{\xi}(\chi)=\pm1}}
 S(\lambda(\chi))
\otimes
\bigotimes_{\substack{\chi \in G^{**}_{\xi}\\
 \nu_{2}^{\xi}(\chi)=0}}
 S^{\lambda(\chi),\lambda^*(\chi)}{(\chi,\xi)}\uparrow^{SG_{2n}}_{\widetilde{SG}(n(\underline{\lambda}))},$$
 where $
 \lambda^*(\chi)=
 \begin{cases}
  \lambda(\chi), &(\pi= 1,\delta),\\
  \lambda'(\chi),  & (\pi=\iota,\delta \otimes \iota).
 \end{cases}
 $ 
Let $e(\underline{\lambda}) \in \mathbb{C}\widetilde{SG}(n(\underline{\lambda}))$
 be an idempotent such that $S(\underline{\lambda}) \cong \mathbb{C}SG_{2n}e(\underline{\lambda})$.
 The idempotent $e(\underline{\lambda})$ can be described by the follows;
$$e(\underline{\lambda})=\prod_{ \chi \in G^{**}_{\xi}}e(\lambda(\chi)),$$
where $\mathbb{C}SG_{|\lambda(\chi)|}e(\lambda(\chi))=
S(\lambda(\chi))$ for  $\nu_{2}^{\xi}(\chi)=\pm1$
and
$\mathbb{C}SG_{2|\lambda(\chi)|}e(\lambda(\chi))=
S^{\lambda(\chi),\lambda^*(\chi)}(\chi,\xi)$ for  $\nu_{2}^{\xi}(\chi)=0$.
We define a subgroup of $\widetilde{SG}(n(\underline{\lambda}))$ by
$$\widetilde{HG}(n(\underline{\lambda}))
=HG_{n} \cap \widetilde{SG}(n(\underline{\lambda}))=
\prod_{\substack{\chi \in G^{*}\\
 \nu_{2}^{\xi}(\chi)=\pm1}}
 HG_{
 \frac{|\lambda(\chi)|}{2}}
\times
\prod_{\substack{\chi \in G^{**}_{\xi}\\
 \nu_{2}^{\xi}(\chi)=0}}
HG_{|\lambda(\chi)|}.$$
Let $e^{\Theta_{\xi,\pi}}_{n}$ be an idempotent of ${\mathbb C}HG_{n}$ affording to $\Theta_{\xi,\pi}$.
Put $e^{\xi,\pi}_{
\widetilde{HG}(n(\underline{\lambda}))}=\frac{1}{|\widetilde{HG}(n(\underline{\lambda}))|}
\sum_{x \in \widetilde{HG}(n(\underline{\lambda}))}\overline{\Theta_{\xi,\pi}(x)}x$.
 We compute
 \begin{align*}
&e^{\Theta_{\xi,\pi}}_{n}e(\underline{\lambda})(1)
=
\frac{
|\widetilde{HG}(n(\underline{\lambda}))|}
{|HG_{n}|}
e^{\xi,\pi}_{
\widetilde{HG}(n(\underline{\lambda}))
}e(\underline{\lambda})(1)\\
&=
\frac{
|\widetilde{HG}(n(\underline{\lambda}))|
}{|HG_{n}|}
\prod_{\substack{\chi \in G^{*}\\
 \nu_{2}^{\xi}(\chi)=\pm1}}
e^{\Theta_{\xi,\pi}}_{{\frac{|\lambda(\chi)|}{2}}}e(\lambda(\chi))(1)
\times
\prod_{\substack{\chi \in G^{**}_{\xi}\\
 \nu_{2}^{\xi}(\chi)=0}}
e^{\Theta_{\xi,\pi}}_{{|\lambda(\chi)|}}e(\lambda(\chi))(1).
 \end{align*}
Since each $e^{\Theta_{\xi,\pi}}_{m}e(\lambda(\chi))(1)$'s in the equation above is a non-zero element from Proposition \ref{smallinv},
  $e^{\Theta_{\xi,\pi}}_{n}e(\underline{\lambda})$ is nonzero. 
  Therefore $S(\underline{\lambda})$ is an irreducible component of $\Theta_{\xi,\pi}\uparrow_{HG_{n}}^{SG_{2n}}$.
 Proposition \ref{doublecoset1} gives us an upper bound 
  of the dimension of $\mathcal{H}^{\Theta_{\xi,\pi}}(SG_{2n},HG_{n}).$
  In general, the number of distinct irreducible representations in $\Theta_{\xi,\pi}\uparrow_{HG_{n}}^{SG_{2n}}$ 
  is less than $\dim \mathcal{H}^{\Theta_{\xi,\pi}}(SG_{2n},HG_{n})$.
  By combining these facts and  Proposition \ref{equals}, we have the both claims of this theorem.
\end{proof}
\section{ $\Theta_{\xi,\pi}$-Spherical Functions}
The rest part of the paper is devoted to computations of $\Theta_{\xi,\pi}$-spherical functions of our commutative
Hecke algebra $\mathcal{H}^{\Theta_{\xi,\pi}}(SG_{2n},HG_{n})=e^{\Theta_{\xi,\pi}}_{n}{\mathbb C}SG_{2n}e^{\Theta_{\xi,\pi}}_{n}$.
We remark $\Theta_{\xi,\delta}=\Theta_{1,\delta} \otimes \Theta_{\xi,1}$  and
 $\Theta_{\xi,\delta \otimes \iota}=\Theta_{1,\delta} \otimes \Theta_{\xi,\iota}$,
thus we have $\mathcal{H}^{\Theta_{\xi,\delta}}(SG_{2n},HG_{n})\cong \mathcal{H}^{\Theta_{\xi,1}}(SG_{2n},HG_{n})$ and
 $\mathcal{H}^{\Theta_{\xi,\delta \otimes \iota}}(SG_{2n},HG_{n}) \cong \mathcal{H}^{\Theta_{\xi,\iota}}(SG_{2n},HG_{n})$
 (cf. \cite{mac}).
Therefore  we consider the cases of $\pi=1$ and $\iota$.
\begin{theorem}\label{spfnidem}
For $\underline{\lambda} \in P^{**}_{\xi,\pi}(n)$,
the functions 
$$\Omega_{\underline{\lambda}}^{\xi,\pi}=
 \frac
 {|HG_{n}||SG_{2n}|}
 {|HG_{n} \cap \widetilde{SG}(n(\underline{\lambda}))|\dim S(\underline{\lambda})}
e^{\xi,\pi}_{HG_{n}}e(\underline{\lambda}) \in {\mathbb C}SG_{2n}e(\underline{\lambda}).$$
are the $\Theta_{\xi,\pi}$-spherical functions of $(SG_{2n},HG_{n},\Theta_{\xi,\pi})$.
Here $e(\underline{\lambda})$ is an idempotent defined in the proof of Theorem \ref{decomposition}.
\end{theorem}
\begin{proof}
It is clear that $x \Omega_{\underline{\lambda}}^{\xi,\pi}=\Theta_{\xi,\pi}(x)\Omega_{\underline{\lambda}}^{\xi,\pi}$ 
for $x \in HG_{n}$.
We compute
 \begin{align*}
&e^{\xi,\pi}_{HG_{n}}e(\underline{\lambda})(1)
=\frac{|HG_{n} \cap \widetilde{SG}(n(\underline{\lambda}))|}{|HG_{n}|}
e^{\xi,\pi}_{HG_{n} \cap \widetilde{SG}(n(\underline{\lambda}))}e(\underline{\lambda})(1)\\
&=
\frac{|HG_{n} \cap \widetilde{SG}(n(\underline{\lambda}))|}{|HG_{n}|}
\prod_{\substack{\chi \in G^{*}\\
 \nu_{2}^{\xi}(\chi)=\pm1}}
e^{\xi,\pi}_{HG_{\frac{|\lambda(\chi)|}{2}}}e(\lambda(\chi))(1)
\times
\prod_{\substack{\chi \in G^{**}_{\xi}\\
 \nu_{2}^{\xi}(\chi)=0}}
 e^{\xi,\pi}_{HG_{|\lambda(\chi)|}}e(\lambda(\chi))(1)\\
 &=
 \frac{|HG_{n} \cap \widetilde{SG}(n(\underline{\lambda}))|}{|HG_{n}|}
\prod_{\chi \in G^{**}_{\xi}}\frac{\dim S^{\lambda(\chi)}(\chi)}{|SG_{\lambda(\chi)}|}
= \frac{|HG_{n} \cap \widetilde{SG}(n(\underline{\lambda}))|\dim S(\underline{\lambda})}{|HG_{n}||SG_{2n}|}.
 \end{align*}
Therefore $\Omega_{\underline{\lambda}}^{\xi,\pi}(1)=1$.
\end{proof}
In the below, we compute $\Theta_{\xi,\pi}$-spherical functions for two special cases.

First we consider the case of
$\chi=\xi \otimes \overline{\chi}$ i.e., $\nu_{2}^{\xi}(\chi)=\pm1$.
Let $V_{\chi}$ be an irreducible unitary representation of $G$ affording $\chi$. 
Let $v_{1},v_{2},\cdots,v_{d}$ be an orthonormal basis of $V_{\chi}$.
Put an matrix representation $A_{\rho}$ of $G$ on $V_{\chi}$ by $$A_{\rho}(g)v_{j}=\sum_{i=1}^{d}r_{ij}(g)v_{i}.$$ 
We can take a unitary matrix $A=(a_{ij})$ such that $A^* A_{\rho}(g)A=D_{\xi}(g)\overline{A_{\rho}}(g)$. 
By the same computation of  the proof of Proposition \ref{mpf}, we have $A\overline{A}=\nu_{2}^{\xi}(\chi)E
\Leftrightarrow A=\nu_{2}^{\xi}(\chi){}^t\!A$.
%
Put $W(\lambda)=W(\lambda;\chi,\xi,\pi)
=
\begin{cases}
2\lambda, & (\nu_{2}^{\xi}(\chi)=1, \pi=1, \lambda \in P_{n})\\
(2\lambda)', & (\nu_{2}^{\xi}(\chi)=-1, \pi=1, \lambda \in P_{n})\\
D(\lambda),& (\nu_{2}^{\xi}(\chi)=1,\pi=\iota, \lambda \in SP_{n}).\\
D(\lambda)',& (\nu_{2}^{\xi}(\chi)=-1,\pi=\iota, \lambda \in SP_{n}).
\end{cases}$
If $\nu_{2}^{\xi}(\chi)=1$ (resp. $-1$), then $S^{W(\lambda)}$ have an element $v_{\lambda}^{\pi}$ such that $h v_{\lambda}^{\pi}=\pi(h)
v_{\lambda}^{\pi}$ (resp. $h v_{\lambda}^{\pi}=\pi(h)
\delta(h)v_{\lambda}^{\pi}$) for $h \in H_{n}$ and $|v_{\lambda}^{\pi}|=1$ (cf. Proposition\ref{smallinv}). 
 We define an inner product $\langle,\rangle$ of $S^{W(\lambda)}(\chi) \cong V_{\chi}^{\otimes 2n}\otimes S^{W(\lambda)}$ 
by
$$\langle \bigotimes_{j=1}^{2n}{t_{i_{j}}}\otimes t, \bigotimes_{j=1}^{2n}{u_{i_{j}}}\otimes u\rangle=\prod_{j=1}^{2n} 
\langle{t_{i_{j}}},{u_{i_{j}}}\rangle_{V_{\chi}} \times \langle t,u\rangle_{S^{W(\lambda)}},$$
where $t_{i_{j}},u_{i_j} \in V_{\chi}$ and $t,u \in S^{W(\lambda)}$.
%
 Put $\omega_{\xi,\pi}^{\chi,\lambda}({\sigma})=\langle v^{\pi}_{\lambda},\sigma v^{\pi}_{\lambda}\rangle_{S^{W(\lambda)}}$ 
 for $\sigma \in S_{2n}$. 
Put $f=\sum_{i,j}a_{ij} v_{i}\otimes v_{j}$.
 Then direct computations give us  $(g,g:1)f=\xi(g)(\chi)f$ and $(1,1:(12))f=\nu_{2}^{\xi}(\chi)f$.
 %
 Put $x_{n}=(1,\cdots,1,g:[2n])$, where $[2n]=(12\cdots2n)$. 
We compute
\begin{align*}
\langle f^{\otimes n}\otimes v_{\lambda}, x_{n}f^{\otimes n}\otimes v_{\lambda} \rangle 
&=\sum_{\substack{{i_{1},\cdots,i_{n}}\\{j_{1},\cdots,j_{n}}}}
a_{j_{n}i_{1}}\overline{a_{i_{1}j_{1}}}\
a_{j_{1}i_{2}}\overline{a_{i_{2}j_{2}}}
\cdots
(\sum_{k=1}^{n}a_{j_{n-1}k}r_{ki_{n}}(g))\overline{a_{i_{n}j_{n}}} \omega_{\xi,\pi}^{\chi,\lambda}{([2n])}\\
&={\rm tr}( (A \overline{A})^{n-1}(AA_{\rho}(g)\overline{A}))\omega_{\xi,\pi}^{\chi,\lambda}{([2n])}
={\rm tr}( (A \overline{A})^{n-1}(AA_{\rho}(g)A^{-1})(A\overline{A}))\omega_{\xi,\pi}^{\chi,\lambda}{([2n])}\\
&=(\nu_{2}^{\xi}(\chi))^{n-1}{\rm tr} A_{\rho}(g)\omega_{\xi,\pi}^{\chi,\lambda}{([2n])}
=(\nu_{2}^{\xi}(\chi))^{n}\chi(g)\omega_{\xi,\pi}^{\chi,\lambda}{([2n])}
\end{align*}
and
\begin{align*}
\langle f^{\otimes n}\otimes v_{\lambda}, x_{n}f^{\otimes n}\otimes v_{\lambda} \rangle 
=(\dim V_{\chi})^n.
\end{align*}
We repeat the same computations and have
$$
\langle f^{\otimes 2n}\otimes v_{\lambda}, x(\underline{\rho})f^{\otimes 2n}\otimes v_{\lambda} \rangle =
(\nu_{2}^{\xi}(\chi))^{n}
\prod_{R \in G_{**}}
\chi(g_{R})^{\ell(\rho(R))}\omega_{\xi,\pi}^{\chi,\lambda}([\hat{\underline{\rho}}]).
$$
We use a notation $\omega_{\xi,\pi}^{\chi,\lambda}({\hat{\underline{\rho}}})$ instead of  $\omega_{\xi,\pi}^{\chi,\lambda}([\hat{\underline{\rho}}])$
and have
\begin{proposition}\label{2nsp}
The $\Theta_{\xi,\pi}$-spherical functions associated to $S^{W(\lambda)}(\chi)$  are given by
\begin{align*}\Omega^{\chi,\lambda}_{\xi,\pi}(x(\underline{\rho}))=\frac{
(\nu_{2}^{\xi}(\chi))^{n}
\omega_{\xi,\pi}^{\chi,\lambda}({\hat{\underline{\rho}}})}{\dim V_{\chi}^n}
\prod_{R \in G_{**}}
\chi(g_{R})^{\ell(\rho(R))}.
\end{align*}
\end{proposition}
%
Next we consider the case of $\nu_{2}^{\xi}(\chi)=0$.
Let $S_{n}^{o}$ and $S_{n}^{e}$ be symmetric groups on  $\{1,3,5,\cdots,2n-1\}$
and  $\{2,4,6,\cdots,2n\}$ respectively. Put $SG_{n}^{o}=G \wr S_{n}^o$ and $SG_{n}^{e}=G \wr S_{n}^e$.
We consider the irreducible representation $S^{\mu,\mu^*}(\chi,\xi) =S^{\mu}(\chi)\otimes S^{\mu^*}(\xi\otimes \overline{\chi})
\uparrow_{SG_{n}^{o}\times SG_{n}^e}^{SG_{2n}}$, where $
 \mu^*=
 \begin{cases}
  \mu, &(\pi= 1),\\
  \mu',  & (\pi=\iota).
 \end{cases}
 $  We just denote by  $\Delta SG_{n}$ a subgroup
of $SG_{n}^{o} \times SG_{n}^{e}$ 
 defined by $HG_{n} \cap SG_{n}^{o} \times SG_{n}^{e}$.
 Let $T_{n}$ be a subgroup of $SG_{2n}$ defined by $T_{n}=\langle (1,\cdots,1;(2i-1,2i))\mid 1 \leq i \leq n\rangle$.
\begin{proposition}\label{ccc}
For $\eta \in {\mathcal L}(G)$, set $\hat{\eta} \in {\mathcal L}(\Delta G)$
 by $\hat{\eta}(g,g)=\eta(g)$.
Then we have 
\begin{align}\label{tw}
\hat{\eta}\uparrow^{G \times G}_{\Delta G}=\bigoplus_{\chi \in G^*} {\chi} \otimes ({\eta \otimes \overline{\chi} )}.
\end{align}
In particular, $(G \times G,\Delta G, \hat{\eta})$ is a Gelfand triple. The $\hat{\eta}$-spherical functions
$\omega_{\chi}$
 are given by
 $$\omega_{\chi}(x,y)=\frac{\eta(y^{-1})}{\chi(1)}\chi(x^{-1}y).$$
\end{proposition}
\begin{proof}
First, we compute $\displaystyle{\langle \chi \otimes \overline{\chi}\otimes \eta,\hat{\eta} \rangle_{\Delta G}=\frac{1}{|\Delta G|}\sum_{(g,g) \in \Delta G} \chi(g)\overline{\chi(g)}\eta(g)
\overline{\eta(g)}}=\langle \chi, \chi \rangle_{G}=1$. 
Therefore $\hat{\eta}\uparrow^{G \times G}_{\Delta G}$ includes $ {\chi} \otimes ({\eta \otimes \overline{\chi} )}$ as an irreducible component. 
Since  we have $\hat{\eta}\uparrow^{G \times G}_{\Delta G}(1)=|G|$ and
 $\sum_{\chi \in G^*} {\chi}(1) {\eta(1) \overline{\chi} (1)}=|G|$, we obtain (\ref{tw}). 
 Second, we have  (see \cite[pp. 397 ]{mac})
 \begin{align*}
 \omega_{\chi}(x,y)=\frac{1}{|\Delta G|}\sum_{(g,g)\in \Delta G}
 \chi \otimes (\eta \otimes \overline{\chi})(x^{-1}g,y^{-1}g) \eta(g^{-1}).
 \end{align*}
We compute
 \begin{align*}
 \omega_{\chi}(x,y)&=\frac{1}{|\Delta G|}\sum_{(g,g)\in \Delta G}
 \chi (x^{-1}g)\eta(y^{-1}g) \overline{\chi}(y^{-1}g)\eta(g^{-1})\\
 &=\frac{\eta(y^{-1})}{|G|}\sum_{g \in G} \chi(x^{-1}g)\overline{\chi}({y^{-1}g})\\
 &=\frac{\eta(y^{-1})}{|G|}\chi^2(x^{-1}y)=\frac{\eta(y^{-1})}{\chi(1)}\chi(x^{-1}y).
 \end{align*}
\end{proof}
%
From this proposition,  $S^{\mu}(\chi)\otimes S^{\mu^*}(\xi\otimes \overline{\chi})
$ is an irreducible component of 
$$\begin{cases}
S^{(n)}(\xi)\uparrow_{\Delta SG_{n}}^{SG_{n}^{o} \times SG_{n}^{e}} & (\pi=1)\\
S^{(1^n)}(\xi)\uparrow_{\Delta SG_{n}}^{SG_{n}^{o} \times SG_{n}^{e}} & (\pi=\iota).
 \end{cases}$$
Since $HG_{n}=T_{n} \times \Delta SG_{n}$ as a set,
 $e_{n}^{\Theta_{\xi,\pi}}$ enjoys a factorization formula $e_{n}^{\Theta_{\xi,\pi}}=e_{n}^{\pi}e_{n}^{\xi,\pi}=e_{n}^{\xi,\pi}e_{n}^{\pi}$.
 Here  
 ${e_{n}^{\pi}=\frac{1}{|T_{n}|}\sum_{t \in T_{n}}\pi(t)t}$
  and 
 $ {e_{n}^{\xi,\pi}=\frac{1}{|\Delta SG_{n}|}\sum_{x \in \Delta SG_{n}}\overline{\Theta_{\xi,\pi}(x)}x}$.
We define a function on $SG_{2n}$ by
$$F_{\lambda}^{\xi,\pi}
=e_{n}^{\Theta_{\xi,\pi}} e_{\lambda}(\chi) \times e_{\lambda^*}(\overline{\chi}\otimes \xi) e_{n}^{\Theta_{\xi,\pi}}
=e_{n}^{\pi}(e_{n}^{\xi,\pi} e_{\lambda}(\chi) \times e_{\lambda^*}(\overline{\chi}\otimes \xi)e_{n}^{\xi,\pi})e_{n}^{\pi},$$
where  $e_{\lambda}(\chi)$ and  $e_{\lambda^*}(\overline{\chi}\otimes \xi)$ are idempotents corresponding to 
 irreducible $SG_{n}^o$-modules $S^{\lambda}(\chi)$  and
  $SG_{n}^e$-module $S^{\lambda^*}(\overline{\chi}\otimes \xi)$ respectively. 
Proposition \ref{ccc} gives us a fact that a function $\widetilde{F_{\lambda}^{\xi,\pi}}=
\frac{h(\lambda)^{2}|G|^{2n}}{\chi(1)^{2n}}
e_{n}^{\xi,\pi} e_{\lambda}(\chi) \times e_{\lambda^*}(\overline{\chi}\otimes \xi)e_{n}^{\xi,\pi}$ is the
$\begin{cases}
 S^{(n)}(\xi) & (\pi=1)\\
  S^{(1^n)}(\xi) & (\pi=\iota)
 \end{cases}$-spherical function of a Gelfand triple $(SG_{n}^o\times SG_{n}^e,
\Delta SG_{n},\Theta_{\xi,\pi})$, where $\Theta_{\xi,\pi}$ is considered as the restriction of $\Theta_{\xi,\pi}$ to 
$\Delta SG_{n}$.
 For $x\in SG_{2n}$, we have
$$e_{n}^{\pi}\widetilde{F_{\lambda}^{\xi,\pi}}e_{n}^{\pi}(x)=\frac{1}{4^n}\sum_{t,s \in T_{n}}\pi(ts)\widetilde{F_{\lambda}^{\xi,\pi}}(sxt).$$
We set $x_{0}=(1,\cdots,1,g:(1,2,\cdots,2n))$ and $t_{n}=(12)(34)\cdots(2n-1,2n)$.
 Then the following is easy:
$$``sx_{0}t \in SG_{n}^o\times SG_{n}^e \Leftrightarrow s=1,t=(1,\cdots,1:t_{n})\ {\rm or}\ s=(1,\cdots,1:t_{n}), t=1".$$
Put $d_{n}^{\lambda}(\chi)
=\frac{1}{\dim S^{\lambda}(\chi)}$. 
Let $\chi_{\rho}^{\lambda}$
 be the irreducible character of $S_{n}$, indexed by the partition $\lambda$ and evaluated at the 
conjugacy class $\rho$.
Now we can compute (cf. \cite[pp. 149, 4.3.9 Lemma]{JK})
\begin{align*}
e_{n}^{\pi}\widetilde{F_{\lambda}^{\xi,\pi}}e_{n}^{\pi}(x_{0})&=
\frac{\pi(t_{n})
}{4^n}
\{ \widetilde{F_{\lambda}^{\xi,\pi}}(1,\cdots,1,g:(135\cdots2n-1))
+\widetilde{F_{\lambda}^{\xi,\pi}}(1,\cdots,1,g,1:(246\cdots2n))\}\\
&=
\begin{cases}
\frac{d^{\lambda}_{n}(\chi)
}{4^n}
(\xi(g^{-1})\chi(g)+\chi(g^{-1}))\chi^{\lambda}_{(n)} & (\pi=1)\\
\frac{d^{\lambda}_{n}(\chi)
}{4^n}
(\xi(g^{-1})\chi(g)+(-1)^{n-1}\chi(g^{-1}))\chi^{\lambda}_{(n)} & (\pi=\iota).
\end{cases}
\end{align*}
For general $x(\underline{\rho})$, repeating these computations with $x(\underline{\rho})$ in place of $x_0$ establishes the following
proposition.

\begin{proposition}\label{lalasp}
We  obtain $\Theta_{\xi,\pi}$-spherical function of $S^{\lambda,\lambda^*}(\chi,\xi)$
\begin{align*}
\Omega^{\chi,\lambda}_{\xi,\pi}(x(\underline{\rho}))
=\frac{ {\chi^{\lambda}_{\hat{\underline{\rho}}}}}{2^n\dim S^{\lambda}(\chi)}
 \prod_{\substack{R \in G_{**}}}
\prod_{i=1}^{n}
 \left(
 \overline{\xi(g_{R})} \chi(g_R)
 + 
\epsilon_{\pi}^{\rho_{i}(R)-1}\overline{\chi(g_{R})}\right)^{m_{i}(\rho(R))},
\end{align*}
where 
$\epsilon_{\pi}=\begin{cases}
1  & (\pi=1)\\
-1& (\pi=\iota).
\end{cases}$
\end{proposition}
\section{Multi-Partitiversion of  Symmetric Functions}
%
Let $p_{r}(R)$ be a power sum symmetric function
with  variables $(x_{i}(R)\mid i \geq 1)$ for each $R \in G_{**}$
and $\Lambda[G_{**}]={\mathbb C}\langle p_{r}(R)\mid r \geq 1, R \in G_{**}\rangle$.
 We define a subalgebra $\Lambda^{\xi,\pi}_{G}$ of $\Lambda[G_{**}]$
by 
$$\begin{cases}\Lambda^{\xi,1}_{G}={\mathbb C}\langle p_{r}(R)\mid \  R\ is\ a\ complex\ or\  \xi \not\equiv-1\ on\ R
 \rangle,\\
\Lambda^{\xi,\iota}_{G}={\mathbb C}\langle p_{r}(R)
\mid
r\ is\ odd\ (resp.\ even),\ if\ 
\xi \equiv 1\ (resp. -1)\ on\ a\ real\ R
\rangle.
\end{cases}$$

Set, for $\underline{\rho} \in P_{**}$,
 $$P_{\underline{\rho}}(G_{**})=\prod_{R \in G_{**}}p_{\rho(R)}(R).$$
Change  variables by setting
\begin{align*}
p_{r}(\chi)&=
\sum_{{R  \in G_{**}; {\rm real}
}}
\!\!\!\!
\frac{\overline{\xi(g_{R})}\chi(g_{R})+\epsilon_{\pi}^{r-1} \overline{\chi(g_{R})}}{2\zeta_{g_{R}}}p_{r}(R)
+
\sum_{{R \in G_{**}; {\rm complex}
}}
\!\!\!\!
\frac{\overline{\xi(g_{R})}\chi(g_{R})+\epsilon_{\pi}^{r-1}\overline{\chi(g_{R})}}{\zeta_{g_{R}}}p_{r}(R).
\end{align*}
Here we remark $p_{r}(\chi)=\epsilon_{\pi}p_{r}(\xi \otimes \overline{\chi})$.
The second orthogonality relation gives us that $p_{r}(\chi)$'s generate the ring $\Lambda^{\xi,\pi}_{G}$. 
Put 
${\mathcal H}^{\Theta_{\xi,\pi}}_{n}={\mathcal H}^{\Theta_{\xi,\pi}}(SG_{2n},HG_{n})$.
Set ${\mathcal H}^{\xi,\pi}=\bigoplus_{n \geq 0}{\mathcal H}^{\Theta_{\xi,\pi}}(SG_{2n},HG_{n})$.
Here we consider to be ${\mathcal H}^{\Theta_{\xi,\pi}}_{0}={\mathbb C}$. 
We define a product on ${\mathcal H}^{\xi,\pi}$ by 
$fg=e^{\Theta_{\xi,\pi}}_{m+n}( f \times g) e^{\Theta_{\xi,\pi}}_{m+n} $ for $f \in {\mathcal H}^{\Theta_{\xi,\pi}}_{m}$ 
and $g \in {\mathcal H}^{\Theta_{\xi,\pi}}_{n}$, where $f \times g$ means a natural embedding of ${\mathbb C}S_{2m} \times 
S_{2n}$ in ${\mathbb C}S_{2m+2n}$. 
From Proposition \ref{doublecoset1}, we have a basis $\{e^{\Theta_{\xi,\pi}}x(\underline{\rho})e^{\Theta_{\xi,\pi}} \mid \underline{\rho} \in \bigcup_{n \geq 0} P_{**}^{\xi,\pi}(n)\}$ of ${\mathcal H}^{\xi,\pi}$. 

We define a map ${\mathcal H}^{\xi,\pi}$ onto $\Lambda^{\xi,\pi}_{G}$ by
$$CH_{\pi}:e_{n}^{\Theta_{\xi,\pi}}x(\underline{\rho})e_{n}^{\Theta_{\xi,\pi}} \mapsto \sqrt{
\prod_{R; \ real}
2^{(1-\epsilon_{\pi})
\ell(\rho(R))}}
 P_{\underline{\rho}}(G_{**})$$
 for $\underline{\rho} \in P_{\xi,\pi}(n)$.
Since there exists $a \in \phi(S_{n} \times S_{m})$  such that $a x(\underline{\rho})\times x(\underline{\sigma})a^{-1}=x(\underline{\rho}\cup \underline{\sigma})$ ($|\underline{\rho}|=n$ and $|\underline{\sigma}|=m$) ,
$CH$ is a ring isomorphism between ${\mathcal H}^{\xi,\pi}$ and $\Lambda^{\xi,\pi}_{G}$.
Under our definition, we have
$$CH_{\pi}(f)=CH_{\pi}(\sum_{g\in SG_{2n}}f(g)g)=\sum_{\underline{\rho} \in P_{**}^{\xi,\pi}(n)} |D_{\underline{\rho}}|f(x(\underline{\rho}))CH(e^{\Theta_{\xi,\pi}}x(\underline{\rho})e^{\Theta_{\xi,\pi}})$$
for $f \in  {\mathcal H}^{\Theta_{\xi,\pi}}_{n}$.
\begin{proposition}\label{each}
Let $J_{\lambda}^{(\alpha)}$ be a 
Jack symmetric function indexed by $\lambda$ at the parameter $\alpha$.
$$|HG_{n}|^{-1}CH_{\pi}(\Omega^{\lambda}_{\xi,\pi})=
\begin{cases}
\frac{|G|^n}{(\dim V_{\chi})^n}
J^{(2)}_{\lambda}(\chi), &(\nu_{2}^{\xi}(\chi)=1,\pi=1)\\
\frac{|G|^n}{(\dim V_{\chi})^n}
\tilde{J}^{(\frac{1}{2})}_{\lambda}(\chi), &(\nu_{2}^{\xi}(\chi)=-1,\pi=1)\\
\frac{|G|^n}{(\dim V_{\chi})^n}
\overline{h}_{\lambda} Q_{\lambda}(\chi), &(\nu_{2}^{\xi}(\chi)=\pm1,\pi=\iota\ :\ Q_{\lambda}{\rm\ is \ a\ Schur's\ Q-function})\\
\frac{|G|^n}{(\dim V_{\chi})^n}
h_{\lambda} S_{\lambda}(\chi)& (\nu_{2}^{\xi}(\chi)=0 :\ S_{\lambda}{\rm\ is \ a\ Schur\ function}),
\end{cases}$$
where $h_{\lambda}$ (resp. $\overline{h}_{\lambda}$) denotes the product of  hook lengths (resp. shifted hook lengths)  of  $\lambda$.
Here $\tilde{J}^{(\frac{1}{2})}_{\lambda}$ is defined by $\tilde{J}^{(\frac{1}{2})}_{\lambda}=
\psi_{2}(J_{\lambda}^{(\frac{1}{2})})$ for an isomorphism $\psi:p_{r}\mapsto \frac{1}{2}p_{r}$ on the 
ring of symmetric functions.
\end{proposition}
\begin{proof}
For finite sets $A$ and $B$, 
we set two kinds of power sums $p_{r}(a)$ ($a\in A$) and $p_{r}(b)$ ($b \in B$)
and change variables by setting $p_{r}(a)=\sum_{b \in B} \frac{a_{b}(r)}{\zeta_{b}} p_{r}(b)$.
We consider a symmetric function ${F}_{\lambda}({a})=\sum_{\rho \vdash n}z_{\rho}^{-1}b_{\rho}^{\lambda}
p_{\rho}(a)$  and compute 
\begin{align*}
{F}_{\lambda}({a})
&=\sum_{\rho \vdash n}\frac{b_{\rho}^{\lambda}}
{\prod_{r=1}^{n}m_{r}(\rho)!r^{m_{r}(\rho)}}
\prod_{r=1}^{n}p_{r}(b)^{m_{r}(\rho)}\\
&=\sum_{\rho \vdash n}\frac{b_{\rho}^{\lambda}}
{\prod_{r=1}^{n}r^{m_{r}(\rho)}}
\prod_{r=1}^{n}
\left\{\sum_{\sum_{b}m_{r}(\rho(b))=m_{r}(\rho)}
\frac{1}
{\prod_{b \in
B}m_{r}(\rho(b))!} \times {\prod_{b \in
B}(\frac{a_{b}(r)}{\zeta_{b}} p_{r}(b))^{m_{r}(\rho(b))}}
\right\}
\\
&=\sum_{\rho \vdash n}{b_{\rho}^{\lambda}}
\prod_{r=1}^{n}
\left\{\sum_{\sum_{b}m_{r}(\rho(b))=m_{r}(\rho)}
\frac{1}
{\prod_{b \in
B}r^{m(\rho(b))}m_{r}(\rho(b))!} \times {\prod_{b \in
B}(\frac{a_{b}(r)}{\zeta_{b}} p_{r}(b))^{m_{r}(\rho(b))}}
\right\}
\\
&=
\sum_{\substack{(\rho(b)\mid b \in B)\vdash n,\\ \cup \rho(b)=\rho}}
\sum_{\rho \vdash n}
\frac{{b_{\rho}^{\lambda}}\prod_{b \in B}
\prod_{r=1}^{n}
a_{b}(r)^{m_{r}(\rho(b))}}
{
\prod_{b \in B}
z_{\rho(b)} \zeta_{c}^{\ell(\rho(b))}} 
{\prod_{b \in B}
\prod_{r=1}^{n}
{p_{r}(b)}^{m_{r}(\rho(b))}}.
\end{align*}
In the below, we set $A=G_{\xi}^{**}$, $B=G_{**}$, $\zeta_{b}=\zeta_{g_{R}} (b \in R)$ and 
$a_{g_{R}}(r)=
\begin{cases}
\frac{\overline{\xi(g_{R})}\chi(g_{R})+\epsilon_{\pi}^{r-1} \overline{\chi(g_{R})}}{2}, & (R\ real) \\
{\overline{\xi(g_{R})}\chi(g_{R})+\epsilon_{\pi}^{r-1} \overline{\chi(g_{R})}}, & (R\ complex).
\end{cases}
$
From \cite{mac,stem},
$$
\begin{cases}
J^{(2)}_{\lambda}
=
|H_{n}|\sum_{\rho \in P_{n}}z_{2\rho}^{-1}
\omega^{\chi,{\lambda}}_{\xi,1}([\underline{\hat{\rho}}])p_{\rho} & (\nu^{\xi}_{2}(\chi)=1),\\ 
J^{(\frac{1}{2})}_{\lambda'}=\frac{|H_{n}|}{2^n}\sum_{\rho \in P_{n}}z_{\rho}^{-1}
(-1)^n
\omega^{\chi,{\lambda}}_{\xi,1}([\underline{\hat{\rho}}])
p_{\rho}
 & (\nu^{\xi}_{2}(\chi)=-1),\\
Q_{\lambda}
=
2^ng^{\lambda}\sum_{\rho \in OP_{n}}z_{\rho}^{-1}
\nu_{2}^{\xi}(\chi)^n\omega^{\chi,{\lambda}}_{\xi,\iota}([\underline{\hat{\rho}}])
p_{\rho} & (\nu^{\xi}_{2}(\chi)=\pm 1)\\
S_{\lambda}=\sum_{\rho \in P_{n}}z_{\rho}^{-1}\chi_{\rho}^{\lambda}p_{\rho}.
\end{cases}
$$
%
We consider the cases of  $\nu^{\xi}_{2}(\chi)=\pm1$, i.e. $\chi=\xi \otimes \overline{\chi}$. In this case, we have
\begin{align*}
p_{r}(\chi)&=
\sum_{{R  \in G_{**}; {\rm real}
}}
\!\!\!\!
\frac{(1+\epsilon_{\pi}^r)\overline{\chi(g_{R})} }{2\zeta_{g_{R}}}p_{r}(R)
+
\sum_{{R \in G_{**}; {\rm complex}
}}
\!\!\!\!
\frac{(1+\epsilon_{\pi}^r)\overline{\chi(g_{R})} }{\zeta_{g_{R}}}p_{r}(R).
\end{align*}
From \cite[Proposition 5.15]{mzonal}, we have the order of each double coset $D_{\underline{\rho}}=HG_{n} x(\underline{\rho})HG_{n}$;
$$|D_{\underline{\rho}}|=|HG_{n}|^2
\prod_{\substack{R=C \in G_{**}\\C=C^{-1}}}
\frac
{1}
{z_{2\rho(R)}\zeta_{C}^{\ell(\rho(R))}}
\times
\prod_{\substack{R=C\cup C^{-1} \in G_{**}\\ C \not= C^{-1}}}
\frac{1}
{z_{\rho(R)} \zeta_{C}^{\ell(\rho(R))}}.$$
Using these notations, $F_{\lambda}$ is rewritten as
$$F_{\lambda}=\frac{1}{|HG_{n}|^2}\sum_{|\underline{\rho}|=n} \prod_{ R \in G_{**}}\chi(g_{R})^{\ell(\rho(R))}b_{\hat{\underline{\rho}}}^{\lambda}
|D_{\underline{\rho}}| \prod_{R \in G_{**}}2^{\ell(\rho(R))}P_{\underline{\rho}}(G_{**}).$$
Therefore we  have
$$
\begin{cases}
J^{(2)}_{\lambda}(\chi)=\frac{|H_{n}|(\dim V_{\chi})^n}{|HG_{n}|^2}\sum_{\underline{\rho}
\in P_{**}^{\xi,1}}|D_{\underline{\rho}}|\Omega^{\chi,\lambda}_{\xi,1}(x(\underline{\rho}))P_{\underline{\rho}}
=\frac{1}{|HG_{n}|}\frac{(\dim V_{\chi})^n}{|G|^n}CH_{\pi}(\Omega^{\chi,\lambda}_{\xi,1})  & (\nu^{\xi}_{2}(\chi)=1),\\
\tilde{J}^{(\frac{1}{2})}_{\lambda}(\chi)=\frac{|H_{n}|(\dim V_{\chi})^n}{|HG_{n}|^2}\sum_{\underline{\rho}
\in P_{**}^{\xi,1}}|D_{\underline{\rho}}|\Omega^{\chi,\lambda}_{\xi,1}(x(\underline{\rho}))P_{\underline{\rho}}
=\frac{1}{|HG_{n}|}\frac{(\dim V_{\chi})^n}{|G|^n}CH_{\pi}(\Omega^{\chi,\lambda}_{\xi,1}) & (\nu^{\xi}_{2}(\chi)=-1) ,\\
Q_{\lambda}(\chi)=\frac{2^n g^{\lambda}(\dim V_{\chi})^n}{|HG_{n}|^2}\sum_{\underline{\rho}
\in P_{**}^{\xi,\iota}}|D_{\underline{\rho}}|\Omega^{\chi,\lambda}_{\xi,\iota}(x(\underline{\rho}))
2^{\ell(\hat{\underline{\rho}})}P_{\underline{\rho}}
=\frac{\overline{h}_{\lambda}}{|HG_{n}|}\frac{(\dim V_{\chi})^n}{|G|^n}CH_{\pi}(\Omega^{\chi,\lambda}_{\xi,\iota})& (\nu^{\xi}_{2}(\chi)=\pm1).
\end{cases}
$$
Here $g^{\lambda}$ denotes the number of shifted standard tableaux of shape $\lambda$.
It is fact that $\overline{h}_{\lambda}=\frac{n!}{g^{\lambda}}$ .
We consider $\nu^{\xi}_{2}(\chi)=0$, i.e. $\chi \not=\xi \otimes \overline{\chi}$.
Then we have
$$S_{\lambda}(\chi)=\frac{2^n \dim S^{\lambda}(\chi)}{|HG_{n}|^2}\sum_{\underline{\rho}
\in P_{**}^{\xi,\pi}}|D_{\underline{\rho}}|\Omega^{\chi,\lambda}_{\xi,\pi}(x(\underline{\rho}))P_{\underline{\rho}}
=\frac{h_{\lambda}}{|HG_{n}|}\frac{(\dim V_{\chi})^n}{|G|^n}CH_{\pi}(\Omega^{\chi,\lambda}_{\xi,\pi}).$$
\end{proof}
Now we state the main theorem of this paper.
\begin{theorem}
We recall the notations of Theorem \ref{decomposition}.
For $\lambda \in P^{**}_{\xi,\pi}$, we have
$$|HG_{n}|^{-1}CH(\Omega^{\xi,\pi}_{\underline{\lambda}})=\prod_{{\chi} \in G^{**}_{\xi} } F_{\lambda}(\chi)$$
Here 
$$
F_{{\lambda}}(\chi)=
\begin{cases}
\frac{(\dim V_{\chi})^n}{|G|^n}
J^{(2)}_{\lambda}(\chi), &(\nu_{2}^{\xi}(\chi)=1,\pi=1)\\
\frac{(\dim V_{\chi})^n}{|G|^n}
\tilde{J}^{(\frac{1}{2})}_{\lambda}(\chi), &(\nu_{2}^{\xi}(\chi)=-1,\pi=1)\\
\frac{(\dim V_{\chi})^n}{|G|^n}
\overline{h}_{\lambda} Q_{\lambda}(\chi), &(\nu_{2}^{\xi}(\chi)=\pm1,\pi=\iota)\\
\frac{(\dim V_{\chi})^n}{|G|^n}
h_{\lambda} S_{\lambda}(\chi)& (\nu_{2}^{\xi}(\chi)=0),
\end{cases}
$$
\end{theorem}
\begin{proof}
In Theorem \ref{spfnidem}, we have already written  down $\Theta_{\xi,\pi}$-spherical function in the form of
product of idempotents. 
Since $e^{\Theta_{\xi,\pi}}_{n}
=
e^{\Theta_{\xi,\pi}}_{n}e^{\Theta_{\xi,\pi}}_{HG_{n}\cap \widetilde{SG}(n(\underline{\lambda}))}$, we have 
\begin{align*}
e^{\Theta_{\xi,\pi}}_{n}e(\underline{\lambda})e^{\Theta_{\xi,\pi}}_{n}
=
  \dim S(\underline{\lambda})\frac{|\widetilde{SG}(n(\underline{\lambda}))|}{|SG_{2n}|}\times
 \frac{1}{|\widetilde{SG}(n(\underline{\lambda}))|}e^{\Theta_{\xi,\pi}}_{n} 
 \prod_{\chi \in G_{\xi}^{**}}\Omega_{\xi,\pi}^{\chi,\lambda(\chi)}e^{\Theta_{\xi,\pi}}_{n}.
\end{align*}
In the above, we use an general formula of $\phi$-spherical functions $\omega_{\chi}$
of a 
Gelfand triple $(G,H,\phi)$;  $\omega_{\chi}=\frac{|G|}{\chi(1)}e_{\chi}e_{\phi}$.
Therefore we have
$$\Omega_{\underline{\lambda}}^{\xi,\pi}=\frac{|HG_{n}|}{|HG_{n}\cap \widetilde{SG}(n(\underline{\lambda}))|}e^{\Theta_{\xi,\pi}}_{n} 
 \prod_{\chi \in G_{\xi}^{**}}\Omega_{\xi,\pi}^{\chi,\lambda(\chi)}e^{\Theta_{\xi,\pi}}_{n}.$$
 Since $CH_{\pi}$ is a ring isomorphism, we have
 $$CH_{\pi}(\Omega_{\xi,\pi}^{\chi,\lambda(\chi)})=|HG_{n}|\prod_{{\chi} \in G^{**}_{\xi} } F_{\lambda(\chi)}(\chi),$$
from Proposition \ref{each}.
\end{proof}
\section{The Case of $G={\mathbb Z}/2{\mathbb  Z}$}
In the last section, we consider a special case; $G={\mathbb Z}/2{\mathbb  Z}=\{1,-1\}$.
Then we remark that $SG_{2n}$ is isomorphic to $H_{2n}$ and 
$HG_{n}$ is isomorphic to a wreath product of Kleinsche Vierergruppe with a symmetric group.
In this case, $G_{**}=\{\{1\},\{-1\}\}$.
Let $\varepsilon$ be the sign representation of $G$ and $\pi$.
For $\pi=1$ or $\delta$ , the irreducible decomposition is 
$$\Theta_{\varepsilon,\pi}\uparrow_{HG_{n}}^{SG_{2n}}=\bigoplus_{|\lambda|=n } S^{\lambda,\lambda}(\chi,\varepsilon)$$
and
non-vanishing double coset is 
$\{HG_{n} x(\underline{\rho}) HG_{n}\mid \underline{\rho}=(\rho(\{1\}),\emptyset)\}$
from Theorem \ref{decomposition}.
From Proposition \ref{lalasp}, the $\Theta_{\varepsilon,1}$-spherical function corresponding to $S^{\lambda,\lambda}(\chi,\varepsilon)$ 
is 
$$\Omega^{1,\lambda}_{\varepsilon,\pi}(x(\rho,\emptyset))=\frac{h_{\lambda} }{2^{n-\ell(\rho)}n!}\chi_{\rho}^{\lambda}.$$ 
For $\pi=\iota$ or $\delta \otimes \iota$ , the irrreducible decomposition is 
$$\Theta_{\varepsilon,\pi}\uparrow_{HG_{n}}^{SG_{2n}}=\bigoplus_{|\lambda|=n } S^{\lambda,\lambda'}(\chi,\varepsilon)$$
and
non-vanishing double coset is 
$\{HG_{n} x(\rho(\{1\}),\rho(\{-1\})) HG_{n}\mid \rho(\{1\})\in OP,\rho(\{-1\})\in EP \}$
from Theorem \ref{decomposition}.
Put $\rho=\rho(\{1\})\cup \rho(\{-1\})$.
From Proposition \ref{lalasp}, the $\Theta_{\varepsilon,1}$-spherical function corresponding to $S^{\lambda,\lambda'}(\chi,\varepsilon)$ 
is 
\begin{align*}
\Omega^{1,\lambda}_{\varepsilon,\pi}(x(\rho(\{1\}),\rho(\{-1\})))=\frac{h_{\lambda} }{2^{n}n!}\chi_{\rho}^{\lambda}
2^{\ell(\rho(\{1\}))}(-2)^{\ell(\rho(\{-1\}))}=\frac{h_{\lambda} }{2^{n-\ell(\rho)}n!}\chi_{\rho}^{\lambda'}.
\end{align*}
In conclusion, tables of   $\Theta_{\varepsilon,\pi}$-spherical function of this case can be obtained by multiplying 
the character table of the symmetric group by some diagonal 
matrices from both sides.

\end{document}